\documentclass[a4paper,reqno]{amsart}

\usepackage{latexsym}
\usepackage[english]{babel}
\usepackage{fancyhdr}
\usepackage[mathscr]{eucal}
\usepackage{amsmath}
\usepackage{mathrsfs}
\usepackage{mathabx}
\usepackage{amsthm}
\usepackage{amsfonts}
\usepackage{amssymb}
\usepackage{amscd}
\usepackage{bbm}
\usepackage{graphicx}
\usepackage{graphics}
\usepackage{latexsym}
\usepackage{color}
\usepackage{subfig}
\usepackage{hyperref}

\newcommand{\virgolette}[1]{``#1''}
\usepackage{scalerel,stackengine}
\stackMath
\newcommand\reallywidehat[1]{%
\savestack{\tmpbox}{\stretchto{%
  \scaleto{%
    \scalerel*[\widthof{\ensuremath{#1}}]{\kern-.6pt\bigwedge\kern-.6pt}%
    {\rule[-\textheight/2]{1ex}{\textheight}}
  }{\textheight}%
}{0.5ex}}%
\stackon[1pt]{#1}{\tmpbox}%
}

\DeclareMathOperator*{\argmax}{arg\,max}

\newcommand{\tr}{\mathrm{Tr}}

\newcommand{\Var}{\mathbb Var}

\theoremstyle{plain}
\newtheorem{theorem}{Theorem}[section]
\newtheorem{lemma}[theorem]{Lemma}

\newtheorem{hyp}{Assumption}

\theoremstyle{definition}

\newtheorem{remark}[theorem]{Remark}
\newtheorem*{remark*}{Remark}

\numberwithin{equation}{section}

\begin{document}

\title[Gibbs sampler for almost exchangeable data]{Rates of convergence for Gibbs sampling in the analysis of almost exchangeable data}

\author[B.~Gerencs\'er]{Bal\'azs Gerencs\'er}
\address[B.~Gerencs\'er]{Alfréd Rényi Institute of Mathematics \\ Reáltanoda utca 13-15, Budapest 1053 (HU) and E\"otv\"os Lor\'and University, Department of Probability and Statistics,  P\'azm\'any P\'eter s\'et\'any 1/C, Budapest 1117 (HU)}
\email{gerencser.balazs@renyi.hu}
\author[A.~Ottolini]{Andrea Ottolini}
\address[A.~Ottolini]{Department of Mathematics, Stanford University \\ 450 Jane Stanford Way, Stanford CA 94305 (USA).}
\email{ottolini@stanford.edu}


\begin{abstract}
Motivated by de Finetti's representation theorem for almost exchangeable arrays,
we want to sample $\mathbf p \in [0,1]^d$ from a distribution with density proportional to $\exp(-A^2\sum_{i<j}c_{ij}(p_i-p_j)^2)$, where $A$ is large and $c_{ij}$'s are non-negative weights. 
We analyze the rate of convergence of a coordinate Gibbs sampler used to simulate from these measures. We show that for every non-zero fixed matrix $C=(c_{ij})$, and large enough $A$, mixing happens in $\Theta(A^2)$ steps in a suitable Wasserstein distance. The upper and lower bounds are explicit and depend on the matrix $C$ through few relevant spectral parameters.
\end{abstract}

\date{\today}

\subjclass[]{}

\thanks{}
\maketitle


\section{Introduction}
\subsection{Motivation.} The task of quantifying how long a Markov chain must be run to be close to its limiting stationary distribution has a large emerging literature, see e.g. the monographs \cite{levin2017markov} or \cite{meyn2012markov}.
However, the examples that illustrate the theory are often \emph{ad hoc} chains, far from the kind of instances used in day to day scientific work. In many practical cases instead, as Diaconis \cite{aldous2013another} said on Markov chain Monte Carlo (MCMC) methods, \\

\virgolette{\emph{if you take any application of MCMC in a real problem and ask if we theoreticians can give a sensible answer to a practitioner about “how long...,” then we can’t}}.\\ \\
The example studied here was suggested in a work of de Finetti \cite{de1972probability} to get multivariate distributions on $[0,1]^d$, $d\geq 2$, to be used in Bayesian estimation of $d$ binomial random variables. This problem arises in connection with various classical statistical problems, such as drug effectiveness among inhomogeneous populations, as surveyed in \cite{Bacallado2015}. de Finetti was particularly interested in understanding situations where $\bold p=(p_1,\ldots p_d)$, the parameter to be estimated, is approximately concentrated around the main diagonal
\begin{align*}
    \mathcal D:=\{p\bold 1: p\in [0,1]\}\subset [0,1]^d, \quad \bold 1:=(1,\ldots,1).
\end{align*}
This situation is referred to as \emph{almost exchangeability}, as it captures the idea that the $d$ binomial random variables are almost indistinguishable. de Finetti's choice of a prior was
\begin{equation}\label{stationary}
    \pi_{A, C}(\bold p)=\frac{1}{Z} \exp\Big\{-A^2\sum_{i<j}c_{ij}(p_i-p_j)^2\Big\}d\bold p, \quad \bold p\in [0,1]^d,
\end{equation}
where $A$ is a large positive scaling constant, $C:=(c_{ij})_{i,j=1}^d$ is a symmetric matrix with non-negative entries that sum to $1$, and $Z=Z(A,C)$ is a normalizing constant. It is convenient to think of $C$ as a network (weighted graph) with no loops, where for $i\neq j$, the weight $c_{ij}$ captures the \emph{a priori} likelihood that $p_i$ and $p_j$ are close to each other. As an example, if all the off-diagonal weights are the same, $C$ is the (re-scaled) adjacency graph of a complete graph on $d$ vertices. If the graph is connected and $A$ is large, the support of \eqref{stationary} is indeed close to the diagonal $\mathcal D$, thus capturing de Finetti's idea of parameters $p_i$'s which are roughly the same. \\ \\
To understand the effects of the various parameters and the choice of the graph, it is natural to draw samples from the prior and look at how the various choices affect the measure \eqref{stationary}. Despite its innocent aspect, that closely resemble a Discrete Gaussian Free Field (see e.g.\,\cite{Lyons2016}), the constraint $\bold p\in [0,1]^d$ poses some challenges in the computation of the normalizing constant as explained in \cite{Genz1992}.
In \cite{Bacallado2015}, the authors used a Gibbs sampling procedure to understand the measures given by \eqref{stationary}. We now turn to understand the mixing dynamics of this process.
\subsection{A brief overview of Gibbs sampling} The Gibbs sampler -- also known as Glauber dynamics -- is a Markov chain whose single step consists in updating a randomly chosen coordinate according to the conditional distribution given all the others (a more precise definition will be given in Section \ref{sectionresult}). Introduced in \cite{Geman1984}, the Gibbs sampler is now used in virtually all fields of science (see e.g. \cite{1995}). Despite its widespread use, sharp bounds on its rate of convergence have been obtained only in few cases. Results in discrete state spaces are known for the Ising model \cite{berger2005glauber,ding2009mixing,levin2010glauber}, graph colouring \cite{vigoda2000improved,frieze2007survey}, and some integrable models \cite{Diaconis2008b, Khare2009}, via a combinations of spectral and coupling techniques. In continuous state spaces, there are fewer results available, with exceptions such as \cite{Smith2014, Caputo2020}. The mainstream approach via drift and minorization \cite{Rosenthal1995, Jones2001} tends instead to underestimate possible diffusive behaviors in the chain.  \\ \\
In the case of the measure $\pi_{A,C}$ of \eqref{stationary}, the Gibbs sampler consists in replacing a random coordinate in the hypercube with a weighted average of the others, and then perturbing it with a (truncated) normal noise. A drift and minorization conditions leads to an upper bound exponentially bad in the parameter $A$. In \cite{Gerencsr2019}, the first author was able to show that, for $d=2$, $\Theta(A^2)$ steps are necessary and sufficient for this chain to mix in total variation. The chain has a diffusive behavior without exhibiting cutoff (in the sense of \cite{Diaconis1996}), and the projection of the walk onto the main diagonal of the square has a Brownian-like behavior. \\ \\This work aims to generalize the result of \cite{Gerencsr2019} to an arbitrary $d$ and a general matrix $C$, showing how the mixing time, for $A$ large, is still governed by the meandering along the diagonal. While our results are obtained in Wasserstein distance, we conjecture similar bounds in total variation to hold as well. 

\subsection{Structure}
The rest of the paper is structured as follows: in Section \ref{sectionresult} we specify the model and state our main results. Section \ref{toolsection} deals with technical ingredients of the proof: various moment bounds to handle the non-homogeneous nature of the Gibbs sampler, a monotone and contracting grand-coupling, and a control on large deviations of the walk away from the diagonal. In Section \ref{diffusivesection} we verify the core block for the result, an approximate diffusive behavior of a suitable projection of the walk onto the main diagonal.
We then collect all elements to complete the proof of the main results on the mixing time in Section \ref{sectionproof}, and point out open questions and future directions in Section \ref{sectionfuture}.

\section{Model and result}\label{sectionresult}
\subsection{The role of the graph structure}
As explained in the introduction, one can think of $C$ as a network, its entries being the weights on a graph with $d\geq 2$ vertices. It is convenient to start assuming the following.
\begin{hyp}\label{assumption}
The network $C$ is connected with unit total weight. 
\end{hyp}
The normalizing assumption is not restrictive because of the presence of the scaling parameter $A$, and later on we will see how to remove the connectivity constraint. Altogether, Assumption \ref{assumption} guarantees
\begin{equation}\label{nondegenerate}
c_{i}:=\sum\limits_{j\neq i} c_{ij}>0, \quad \sum\limits_{i\in [d]}c_i=1.
\end{equation}
Equivalently, the diagonal matrix $D$ defined by $D_{i,i}:=c_i$ is non-singular with $\tr(D)=1$. The vector $\bold c=(c_i)_{i\in [d]}$ is nothing but the stationary distribution of the random walk on the network. 
\\ \\
The main goal of this paper is to describe the mixing property of the Gibbs sampler in terms of few spectral parameters of the network $C$. To this aim, for $\bold p\in \mathbb R^d$, we introduce the quantities
\begin{equation}\label{LAP}
\widehat{\bold p}:=D^{-1}C\bold p, \quad \Delta\bold p:=(I-D^{-1}C)\bold p=\bold p-\widehat{\bold p}.
\end{equation}
The vector $\hat {\bold p}$ represents the single update distribution of a random walker on the network $C$ who starts with distribution $\bold p$. $\Delta$ is the Laplacian of the network, which is a self-adjoint operator with respect to the inner product induced by the stationary distribution $\bold c$, i.e.,
\begin{equation}\label{inner}
\langle \Delta\bold p,\bold q\rangle=\langle \bold p,\Delta \bold q\rangle,\quad  \quad \langle \bold p,\bold q\rangle:=\bold p^{T}D\bold q=\sum\limits_{i\in [d]}c_ip_iq_i.
\end{equation}
Here, we dropped the dependence on $D$ in the scalar product notation for convenience. Under Assumption \ref{assumption}, the Perron-Frobenius theorem is in force for $D^{-1}C$. As a consequence, the Laplacian has $d$ real eigenvalues that satisfy
\begin{align*}
    0=\lambda_1<\lambda_2\leq\ldots\leq \lambda_d\leq 2,
\end{align*}
with the one dimensional (right) eigenspace, corresponding to $\lambda_1=0$, being spanned by the normalized eigenvector $\bold 1$. 
The connection between the Laplacian and the measure in \eqref{stationary} comes from the identity 
\begin{equation}\label{hereislap}
\sum_{i<j}c_{ij}(p_i-p_j)^2=\langle \bold p,\Delta \bold p\rangle. 
\end{equation}
Our result will be stated in terms of the quantities 
\begin{equation}\label{mainquantities}
\lambda:=\lambda_2,\quad  \gamma:=\max(|1-{\lambda_2}|,|1-{\lambda_d}|),\quad \beta:=\min_{i\in [d]}c_i.
\end{equation}
For a heuristic understanding, recall that $\lambda$ is intimately related to the \emph{connectivity} of the network, $\gamma$ to the \emph{mixing properties} of the associated random walk, and $\beta$ to the \emph{accessibility} of the stationary distribution. In particular, since these parameters can be estimated with various probabilistic and geometric tools, and are stable under small perturbations of the weights $c_{ij}$, our bounds will be robust. 

\subsection{The Gibbs sampler}
In our context, the Gibbs sampler is a Markov chain on $[0,1]^d$, with stationary distribution $\pi_{A,C}$. Given the current state $\bold p=(p_1,\dots,p_d)$, it updates according to the following rules:
\begin{enumerate}
    \item sample $I\in [d]$ uniformly at random;
    \item conditioned on $I=i$, update the $i$-th coordinate of $\bold p$ according to the conditional distribution $\pi_{A,C}(p_1,\ldots, p_{i-1},\cdot,p_{i+1},\ldots p_d)$, where $\pi_{A,C}$ is the measure in \eqref{stationary}.
\end{enumerate}
Conditional on the choice of index $i$ in the first step, the update in the second step can be easily derived combining \eqref{stationary} and \eqref{hereislap}, and gives
\begin{equation}\label{updateonecoord}
    p_i\rightarrow \hat p_i+\epsilon_i (\sigma_i^2,\hat p_i), \quad \sigma_i^2=\sigma_i^2(A,C):=\frac{1}{2A^2c_i}.
\end{equation}
Here, $\epsilon(\sigma^2, p)$ denotes a normal random variable with zero mean, variance $\sigma^2$ and conditioned to lie in $[-p,1-p]$. If $c_i=\infty$ (i.e., $\sigma_i=\infty$) we follow the convention that $\epsilon(\infty, p)$ is uniformly distributed on $[-p,1-p]$. Sampling from a truncated normal random variable can be done efficiently (see e.g.\,\cite{Chopin2010}), and the usefulness of the Gibbs sampler is then related to its rate of convergence to stationarity. \\ \\ Formula \eqref{updateonecoord} shows that the walk consists of a \emph{non-local} move (replacing a random coordinate with a weighted average of all the others), and a \emph{non-homogeneous diffusive} move (adding a truncated normal noise). The non-locality creates some difficulties in approximating the walk with a diffusion, and the lack of homogeneity requires to take care of the effect of the boundary of the hypercube. Intuitively, for large $A$ one expects the behavior of the chain at the two corners $\bold 0$ and $\bold 1$ -- where coordinates are clustered together and in the vicinity of the boundary of the hypercube --  to represent the main challenge.
\begin{remark}\label{IPS}
One can think of the Gibbs sampler as the dynamics on an interacting system with $d$ particles constrained to lie in $[0,1]$. In this language, our analysis concerns the mixing time in the \emph{low temperature} regime before taking a thermodynamic limit (i.e. for a fixed number of particles).
\end{remark}
We now write \eqref{updateonecoord} in a more compact form. Let $\Pi_I$ be the (random) projection onto the $I$-th basis vector, where $I$ is chosen uniformly at random in $[d]$. For $\bold p\in [0,1]^d$ let $\boldsymbol{\epsilon}(\Sigma,\bold p)$ be a normal vector with mean $\bold 0$, covariance matrix $\Sigma$ and conditioned to lie in $[-\bold p,\bold 1-\bold p]\subset \mathbb R^d$. Owing to \eqref{LAP} and \eqref{updateonecoord}, conditioned on $I=i$,  we can write the update of one step of the Gibbs sampler as 
\begin{equation}\label{totalstep}
    \bold p\rightarrow \bold p+\Pi_i\Big[-\Delta \bold p+\boldsymbol{\epsilon}(\Sigma,\hat{\bold p})\Big], \quad \Sigma=\Sigma(A,C):=\frac{D^{-1}}{2A^2}.
\end{equation}
We will denote by $K_{A,C}^{*k}(\bold p)$ the law of a random variable obtained after performing $k$ steps of the Gibbs sampler, starting from the point $\bold p$. By $K_{A,C}(\bold p)$, we will mean the single step update.
\subsection{Main results}
Our result is expressed in term of $\infty$-Wasserstein distance. If $\mu, \nu$ are Borel probability measures on $[0,1]^d$, we define
\begin{equation}\label{metricwecare}
    d_{\infty}(\mu,\nu):=\min_{\bold p\sim \mu, \bold q\sim \nu}\mathbb E\Big[\|\bold p-\bold q\|_{\infty}\Big].
\end{equation}
Here, the minimum is taken over all couplings of $\bold p$ and $\bold q$ with marginals $\mu$ and $\nu$, respectively, and $\|\bold p-\bold q\|_{\infty}=\sup_{i\in [d]}|p_i-q_i|$. The existence of a minimizer in \eqref{metricwecare} is a known fact (see e.g. Theorem $4.1$ in \cite{villani2008optimal}).
We can now state our main theorem. 
\begin{theorem}\label{MAINTHEOREM}
Let $C$ be a network on $d\geq 2$ vertices satisfying Assumption \ref{assumption}, and let $\lambda, \gamma, \beta$ be defined as in \eqref{mainquantities}. Then, for any $\delta>0$, there exists $m=m(\delta), M=M(\delta)$ such that, for all $A\geq A^*(\delta,\beta,\lambda)$, one has:
\begin{itemize}
    \item if $k\leq m\frac{\lambda}{\gamma}A^2$, then
    \begin{align*}
        \sup_{\bold p\in [0,1]^d}d_{\infty}(K_{A,C}^{*k}(\bold p),\pi_{A,C})\geq \frac{1}{4}-\delta; 
    \end{align*}
    \item if $k\geq M dA^2$, then 
    \begin{align*}
    \sup_{\bold p\in [0,1]^d}d_{\infty}(K_{A,C}^{*k}(\bold p),\pi_{A,C})\leq \delta.
    \end{align*}
\end{itemize}
\end{theorem}
\begin{remark}
Considering the lower bound, one can see that $\frac{1}{4}=d_{\infty}(\delta_{\frac{1}{2}\bold 1}, \delta_{\mathcal D})$, where $\delta_{\mathcal D}$ is the uniform measure on the main diagonal. We conjecture that it is possible to replace it with $\frac{1}{2}=d_{\infty}(\delta_{\bold 0}, \delta_{\mathcal D})$.
\end{remark}
If the network is disconnected, the measure in \eqref{stationary} becomes a product measure over the different components. Similarly, if two coordinates in different connected components are selected, the corresponding moves of the Gibbs sampler are independent of each other. By means of some simple concentration inequalities, this leads to the following generalization of Theorem \ref{MAINTHEOREM}.
\begin{theorem}\label{MAINTHEOREMDISC}
Let $C$ be a network with at least one positive weight. Then for every $\delta>0$, there exists $\alpha=\alpha(C,\delta)$, $\alpha'=\alpha'(C,\delta)$ such that, for all $A\geq A^*(C,\delta)$, one has:
\begin{itemize}
    \item if $k\leq \alpha A^2$, then
    \begin{align*}
        \sup_{\bold p\in [0,1]^d}d_{\infty}(K_{A,C}^{*k}(\bold p),\pi_{A,C})\geq \frac{1}{4}-\delta;
    \end{align*}
    \item if $k\geq \alpha' A^2$, then
    \begin{align*}
        \sup_{\bold p\in [0,1]^d}d_{\infty}(K_{A,C}^{*k}(\bold p),\pi_{A,C})\leq \delta.
    \end{align*}
\end{itemize}
\end{theorem}

Informally, our results shows that for a fixed non-zero network $C$, similarly to the two dimensional case analyzed in \cite{Gerencsr2019}, the Gibbs sampler mixes in about $A^2$ steps with no cutoff. As an example, consider the case where the underlying network is the complete graph with uniform weights. In this case, the Laplacian is maximally symmetric on the orthogonal complement of the vector $\bold 1$. In particular, 
\begin{align*}
\lambda=\frac{d}{d-1}, \quad \gamma=\frac{1}{d-1}, \quad \frac{\lambda}{\gamma}=d.
\end{align*}
Therefore, according to Theorem \ref{MAINTHEOREM}, $\Theta(dA^2)$ are necessary and sufficient to mix. The result is in accordance with Figures \ref{fig:variancegrowth} and \ref{fig:hitallrange}, that illustrate the evolution of \eqref{barycenterdef} and \eqref{relevantstopping}: during the proof, these statistics will be shown to be the relevant ones in determining the mixing time.\\ \\
For more general networks, quantities like $\lambda$ and $\gamma$ can be also estimated by geometric-analytic methods, such as Cheeger inequalities \cite{cheeger1969lower} or Poincair\'e/Nash-Sobolev inequalities \cite{diaconis1991geometric}. Our proof gives explicit bounds on the size of $A^*$ in terms of these quantities. For the usefulness of Theorem \ref{MAINTHEOREM} when dealing with networks other than the complete graph, as well as the issues with dealing with total variations, we refer the reader to Section $6$.
\section{Technical components}\label{toolsection}
The quantities defined in \eqref{mainquantities} allow to relate different useful quadratic forms. More precisely, define
\begin{align*}
    \mathcal D^{\perp}:=\{\bold p\in \mathbb R^d: \langle \bold p,\bold 1\rangle =0\}.
\end{align*}
Then, for $\bold p\in \mathcal D^{\perp}$ the spectral theorem for $\Delta$ and $D^{-1}C$ entails
\begin{equation}\label{spectralbounds}
    \langle \bold p, \Delta \bold p\rangle \geq \lambda \langle \bold p,\bold p\rangle, \quad\quad  \langle \widehat {\bold p},\widehat {\bold p}\rangle=\langle(I-\Delta)\bold p,(I-\Delta)\bold p \rangle\leq \gamma^2 \langle \bold p,\bold p\rangle.
\end{equation}
In particular, since $\Delta^{\frac{1}{2}}\bold p\in\mathcal D^{\perp}$ for all $\bold p$, the first bound readily implies
\begin{equation}\label{lambdastuff}
    \langle \Delta\bold p,\Delta \bold p \rangle =\langle \Delta^{\frac{1}{2}}\bold p,\Delta\Delta^{\frac{1}{2}}\bold p\rangle \geq \lambda \langle \bold p, \Delta\bold p \rangle,\quad \bold p\in\mathbb R^d.
\end{equation}
The following bound will also be useful. 
\begin{lemma}\label{cutefact}
If $C$ satisfies Assumption \ref{assumption}, then
\begin{align*}
    \max_{i\in [d]} c_i\leq \gamma. 
\end{align*}
\end{lemma}
\begin{proof}
Without loss of generality, assume $\max\limits_{i\in [d]} c_i=c_1=:c$.  Consider the vector 
\begin{align*}
\bold p=\Big(\sqrt\frac{1-c}{c},-\sqrt\frac{c}{1-c}, \dots,-\sqrt\frac{c}{1-c}\Big).
\end{align*}
An easy computation shows $\langle \bold p,\bold p\rangle=1, \langle \bold p,\bold 1\rangle=0$, as well as
\begin{align*}
    \langle \bold p,\hat{\bold p}\rangle=-\frac{c}{1-c}.
\end{align*}
Therefore, we obtain
\begin{align*}
    \max_{i\in [d]} c_i=c\leq \Big|-\frac{c}{1-c}\Big|=|\langle \bold p,\hat{\bold p}\rangle|\leq \gamma,
\end{align*}
where the last bound follows from the spectral theorem for $D^{-1}C$ and \eqref{spectralbounds}. 
\end{proof}
\subsection{Truncated normal distributions}
Recall that, for $\sigma>0$ and $p\in [0,1]$, we denote by $\epsilon(\sigma^2,p)$ a normal random variable with zero mean, variance $\sigma^2$, which is conditioned to lie in $[-p,1-p]$. When $\sigma$ is small, the truncation has little effect as long as $p$ is in the inner part of $[0,1]$. We quantify this effect by means of uniform bounds on the moments of $\epsilon(\sigma^2,p)$. Here and after, we will always denote by $f(x), F(x)$ the density and distribution function, respectively, of a standard normal random variable, that will be denoted by $Z$. 
\begin{lemma}\label{truncatedstuff}
Let $\epsilon(\sigma^2,p)$ be a truncated normal random variable. Then the following facts hold:
\begin{enumerate}
    \item for all $p\in [0,1]$ and $\sigma>0$, one has  $\epsilon(\sigma^2,p)\stackrel{d}{=}- \epsilon(\sigma^2,1-p)$;
    \item for any $\sigma>0$ and $p\in [0,\frac{1}{2}]$,
    \begin{equation}\label{explicitmean}
    0\leq E(\epsilon(\sigma^2,p))\leq 2\sigma e^{-\frac{p^2}{2\sigma^2}};
    \end{equation}
    \item for any $\sigma>0$ and $p\in [\frac{1}{2},1]$,
    \begin{align*}
    -2\sigma e^{-\frac{(1-p)^2}{2\sigma^2}}\leq E(\epsilon(\sigma^2,p))\leq 0;
    \end{align*}
    \item For any $\sigma>0, p\in [0,1]$,
    \begin{equation}\label{varup}
        \Var(\epsilon(\sigma^2, p))\leq \sigma^2, \quad 
        \mathbb E\Big(\epsilon^2(\sigma^2,p)\Big)\leq 5\sigma^2. 
    \end{equation}
    \item There exists $\rho\in (0,1)$ such that, for all $\sigma\leq 1$ and all $p\in [0,1]$, one has
    \begin{equation}\label{vardown}
        \Var(\epsilon(\sigma^2,p))\geq \rho \sigma^2.
    \end{equation}
\end{enumerate}
\end{lemma}
\begin{proof}
Part $(1)$ follows from $Z\stackrel{d}{=}-Z$. Standard computations (see e.g. \cite{burkardt2014truncated}) show 
\begin{align*}
    \mathbb E(\epsilon(\sigma,p))=\sigma\frac{f\Big(\frac{p}{\sigma}\Big)-f\Big(\frac{1-p}{\sigma}\Big)}{F\Big(\frac{1-p}{\sigma}\Big)+F\Big(\frac{p}{\sigma}\Big)-1}, 
\end{align*}
which gives immediately the lower bound in \eqref{explicitmean}. As for the upper bound, if $\sigma\geq \frac{19}{20}, p\in [0,\frac{1}{2}]$ there is nothing to prove since the right side in \eqref{explicitmean} -- which is monotone in $\sigma$ -- is greater than one while $|\epsilon(\sigma^2,p)|\leq 1$. If $\sigma\leq \frac{19}{20}$ we use
\begin{align*}
    E(\epsilon(\sigma,p))\leq \sigma\frac{1}{\sqrt{2\pi}} \frac{e^{-\frac{p^2}{2\sigma^2}}}{F\Big(\frac{1}{2\sigma}\Big)-\frac{1}{2}}\leq 2\sigma e^{-\frac{p^2}{2\sigma^2}},
\end{align*}
where in the last step we bound
\begin{align*}
    \sqrt{2\pi}\Big[F\Big(\frac{1}{2\sigma}\Big)-\frac{1}{2}\Big]\geq \int_{0}^{\frac{10}{19}}e^{-\frac{x^2}{2}}dx\geq\frac{1}{2}.
\end{align*}
This proves part $(2)$. Combining part $(1)$ and part $(2)$ we get part $(3)$. Another standard computation (see again \cite{burkardt2014truncated}) gives
\begin{equation}\label{longcomplicated}
    \frac{\Var(\epsilon(\sigma^2, p))}{\sigma^2}=1-\frac{1}{\sigma}\frac{pf\Big(\frac{p}{\sigma}\Big)+(1-p)f\Big(\frac{1-p}{\sigma}\Big)}{F\Big(\frac{p}{\sigma}\Big)+F\Big(\frac{1-p}{\sigma}\Big)-1}-\Bigg(\frac{f\Big(\frac{p}{\sigma}\Big)-f\Big(\frac{1-p}{\sigma}\Big)}{F\Big(\frac{p}{\sigma}\Big)+F\Big(\frac{1-p}{\sigma}\Big)-1}\Bigg)^2.
\end{equation}
The first bound in $(4)$ follows immediately for any $\sigma>0$, while the second bound is obtained combining $(2)$ with the upper bound on the variance. As for $(5)$, owing to part $(1)$ we can assume $p\in[0,\frac{1}{2}]$ without loss of generality. Since, for fixed $p\neq 0$ and $\sigma\rightarrow 0$, the right side in \eqref{longcomplicated} converges to one, we only need to show that it is bounded away from $0$ as $\sigma\rightarrow 0$ and for $p$, say, in $[0,\frac{1}{10}]$. In this case, \eqref{longcomplicated} leads to the asymptotic relation
\begin{align*}
    \frac{\Var(\epsilon(\sigma^2,p))}{\sigma^2}\sim 1-x\frac{f(x)}{F(x)}-\Big(\frac{f(x)}{F(x)}\Big)^2, \quad x:=\frac{p}{\sigma}.
\end{align*}
Since, for all $x\geq 0$, we have the bounds $\frac{f(x)}{F(x)}\leq \sqrt{\frac{2}{\pi}}$ and $F(x)\geq \frac{1}{2}$, we obtain
\begin{align*}
    x\frac{f(x)}{F(x)}+\Big(\frac{f(x)}{F(x)}\Big)^2\leq \sqrt{\frac{2}{\pi}}\Big(x+\sqrt{\frac{2}{\pi}}\Big)e^{-\frac{x^2}{2}}<1,
\end{align*}
where the last bound follows by computing the explicit maximum.
\end{proof}
The upper bound in \eqref{explicitmean} shows that the drift for the truncated normal noise induced by the boundaries decays rapidly when $p$ is at distance $O\Big(\frac{1}{\sigma}\Big)$ from $0$ or $1$.\\ \\
Combining \eqref{updateonecoord} and \eqref{vardown}, at every step of the Gibbs sampler we get a uniform lower bound on the variance of the truncated normal step given by
\begin{align*}
    \Var(\epsilon(\sigma^2_i,p))\geq \rho \sigma^2_i,
\end{align*}
as long as $A^2\geq\frac{1}{2\beta}$. This is one of the reason for why $A$ has to be large enough in our result.
\subsection{A monotone and contractive coupling} 
Consider two nearby points $p\leq q$ -- say in $[0,\frac{1}{2}]$ -- and perturb them with a truncated normal noise. Because of \eqref{explicitmean}, we may hope to couple them in such a way that they will get even closer, since the one to the left will encounter more drift. Owing to the log-concavity of the normal distribution, we can do better, and maintain their order as well.
\begin{lemma}\label{coupling}
Fix $\delta>0$ and $\sigma\in (0,\infty]$. Let $p,q\in [0,1]$ such that $0\leq p-q\leq \delta$. Then, it is possible to couple two random variables
\begin{align*}
    X_p\stackrel{d}{=}p+\epsilon(\sigma^2,p), \quad X_q\stackrel{d}{=}q+\epsilon(\sigma^2,q)
\end{align*}
in such a way that $0\leq X_p-X_q\leq \delta$.
\end{lemma}
\begin{proof}
If $\sigma=\infty$, then $q+\epsilon(\sigma^2,q)\stackrel{d}{=}p+\epsilon(\sigma^2,p)$ as they are both uniformly distributed on the unit interval. Otherwise, let $\Gamma(p,\cdot)$ the inverse distribution function of the random variable $p+\epsilon(\sigma^2,p)$. For $U$ uniformly distributed on $[0,1]$, define $X_p:=\Gamma(p,U)$. Using the \emph{same} noise $U$ for constructing $X_q$, we have 
\begin{align*}
    X_p\stackrel{d}{=}p+\epsilon(\sigma^2,p), \quad X_q\stackrel{d}{=}q+\epsilon(\sigma^2,q), \quad X_p-X_q=\Gamma(p,U)-\Gamma(q,U).
\end{align*}
In particular, as a result of the mean value theorem, it suffices to show 
\begin{equation}\label{toprove}
    \frac{\partial \Gamma}{\partial p}(p,u)\in [0,1], \quad p,u\in [0,1].
\end{equation}
Let $\Phi(p,\cdot)$ the distribution function of $X_p$, so that 
\begin{align*}
    \Phi(p,u)=\frac{\int_0^{u}f\Big(\frac{x-p}{\sigma}\Big)dx}{F\Big(\frac{1-p}{\sigma}\Big)+F\Big(\frac{p}{\sigma}\Big)-1}
\end{align*}
Consider the identity
\begin{align*}
    \Phi(p,\Gamma(p,u))=u,\quad p,u\in [0,1].
\end{align*}
Differentiating with respect to $p$ and rearranging, we obtain
\begin{align*}
\frac{\partial \Gamma}{\partial p}(p,u)=1-\frac{(1-u)f\Big(\frac{p}{\sigma}\Big)+uf\Big(\frac{1-p}{\sigma}\Big)}{f\Big(\frac{p-\Gamma(p,u)}{\sigma}\Big)}.
\end{align*}
The upper bound in \eqref{toprove} follows immediately. As for the lower bound, we need to show 
\begin{align*}
    G(p,u):=f\Big(\frac{p-\Gamma(p,u)}{\sigma}\Big)-(1-u)f\Big(\frac{p}{\sigma}\Big)-uf\Big(\frac{1-p}{\sigma}\Big)\geq 0, \quad p,u\in [0,1].
\end{align*}
Since $G(p,0)=G(p,1)=0$ for all $p\in [0,1]$, it is enough to show that $\frac{\partial G}{\partial u}(p,u)$ is decreasing in $u$ for all $p,u\in [0,1]$. The claim then follows as $G(p,u)$ is concave in $u$ for each $p$.
Using the definition of $\Gamma(p,u)$ and inverse differentiation rule we get
\begin{align*}
    \frac{\partial G}{\partial u}(p,u)&=\frac{p-\Gamma(p,u)}{\sigma}f\Big(\frac{p-\Gamma(p,u)}{\sigma}\Big)\frac{\partial \Gamma(p,u)}{\partial u}+f\Big(\frac{p}{\sigma}\Big)-f\Big(\frac{1-p}{\sigma}\Big)\\&=\frac{p-\Gamma(p,u)}{\sigma}\Big(F\Big(\frac{1-p}{\sigma}\Big)+F\Big(\frac{p}{\sigma}\Big)-1\Big)+f\Big(\frac{p}{\sigma}\Big)-f\Big(\frac{1-p}{\sigma}\Big).
\end{align*}
Since the only dependence on $u$ is through $\Gamma(p,u)$, which is increasing in $u$, we conclude.
\end{proof}
\begin{remark}
Fix $\Phi$ log-concave on $\mathbb R$. Lemma \ref{coupling} remains valid if $\epsilon(\sigma^2,p)$ is replaced by $\tilde{\epsilon}(p)$ whose density on $[0,1]$ is proportional to $\Phi(\cdot-p)$. 
\end{remark}
We can equip $\mathbb R^d$ with the coordinate-wise partial order, so that $\bold p\leq \bold q$ if and only if $p_i\leq q_i$ for all $i\in [d]$. Owing to Lemma \ref{coupling}, we can couple two walkers starting from different points in a monotone and contractive fashion (with respect to the $\|\cdot \|_{\infty}$ norm). By using the same uniform random variable $U$ given in the proof of Lemma \ref{coupling} to update multiple walkers, we can actually do better and construct a monotone and contractive grand-coupling for an arbitrary number of initial conditions.

\begin{lemma}\label{Coup}
Let $C$ be a network on $d$ vertices, and let $A>0$ and $S$ be any index set. Fix $\delta>0$, and consider a sequence $\{\bold p(s)\}_{s\in S}\subset [0,1]^d$.
Then, we can construct random variables $\{\bold p'(s)\}_{s\in S}\subset [0,1]^d$ on a common probability space, with $\bold p'(s)\sim K_{A,C}(\bold p(s))$ for $s\in S$, so that whenever for some $s,r\in S$ we have 
\begin{equation}\label{grandcoup}
    \bold 0\leq \bold p(s)-\bold p(r)\leq \delta \bold 1,
\end{equation}
then we also have
\begin{align*}
    \bold 0\leq \bold p'(s)-\bold p'(r)\leq \delta \bold 1.
\end{align*}
\end{lemma}
\begin{proof}
Choose the same uniformly random coordinate $i\in [d]$ for updating all the $\bold p(s)$'s. Because of \eqref{grandcoup} and convexity, for a pair $s,r\in S$ of interest above we obtain
\begin{align*}
    0\leq \hat p_i(s)-\hat p_i(r)\leq \delta.
\end{align*}
Finally, we conclude owing to Lemma \ref{coupling}, since we can use the same noise to generate all the truncated normal updates at once in a monotone and contractive fashion.
\end{proof}
\subsection{Clustering on the main diagonal}
While we will show that the mixing time is governed by the diffusion on the diagonal, it is still important to keep control of how spread out the coordinates are. There are two forces in competition with each other: while the averaging process tends to attract coordinates together, the normal noise will play the opposite role. A measure of the clustering of the coordinates is given by $\langle \bold p,\Delta\bold p\rangle$ since it vanishes only when $\bold p\in \mathcal D$ under Assumption \ref{assumption}. The following lemma describes its evolution when $\bold p$ is updated according to the Gibbs sampler. 
\begin{lemma}\label{lemmaexcursion}
Let $C$ be a network on $d$ vertices satisfying Assumption \ref{assumption}, and let $A>0$. For any $k\in \mathbb N$ and $\bold p_0\in [0,1]^d$, let $\bold p(k)\sim K^{*k}_{A,C}(\bold p_0)$. Then
\begin{align*}
    \mathbb E\Big(\langle \bold p(k), \Delta\bold p(k)\rangle \Big)\leq \Big(1-\frac{\lambda}{d}\Big)^k\mathbb E\Big(\langle \bold p_0,\Delta\bold p_0\rangle\Big)+\frac{5d}{2\lambda A^2}.
\end{align*}
In particular, if $\bold p_0\in \mathcal D$, then
\begin{equation}\label{excursionbound}
    \mathbb E\Big(\langle \bold p(k),\Delta\bold p(k)\rangle\Big)\leq \frac{5d}{2\lambda A^2}
\end{equation}
\end{lemma}
\begin{proof}
Write $\bold q':=\bold p(k)$ and $\bold q:=\bold p(k-1)$. According to \eqref{totalstep}, conditional on index $i$ being chosen we can write
\begin{align*}
    \langle \bold q',\Delta\bold q'\rangle-\langle \bold q,\Delta\bold q\rangle &=\left\langle \Pi_i\Big(-\Delta \bold q+\boldsymbol{\epsilon}(\Sigma, \hat{\bold q})\Big), \Delta \bold q+\boldsymbol{\epsilon}(\Sigma,\hat{\bold q})\right\rangle.
\end{align*}
Taking expected values and using $\sum\limits_{i\in [d]}\Pi_i=I$,  
\begin{align*}
    \mathbb E\Big(\langle \bold q', \Delta\bold q'\rangle \Big)-\mathbb E\Big(\langle \bold q, \Delta\bold q\rangle \Big)=\frac{1}{d}\mathbb E\Big[-\langle\Delta \bold q,\Delta \bold q\rangle+ \langle \boldsymbol{\epsilon}(\Sigma,\hat{\bold q}), \boldsymbol{\epsilon}(\Sigma,\hat{\bold q})\rangle\Big].
\end{align*}

Using \eqref{lambdastuff}, \eqref{updateonecoord}, and \eqref{varup}, we obtain
\begin{align*}
    \mathbb E\Big(\langle \bold q', \Delta\bold q'\rangle \Big)\leq \Big(1-\frac{\lambda}{d}\Big)\mathbb E\Big(\langle \bold q, \Delta\bold q\rangle \Big)+\frac{1}{d}\sum_{i\in [d]} \frac{5c_i}{2c_iA^2}=\Big(1-\frac{\lambda}{d}\Big)\mathbb E\Big(\langle \bold q, \Delta\bold q\rangle \Big)+\frac{5}{2A^2}.
\end{align*}
The result then follows by iterating in $k$ the inequality above and bounding
\begin{align*}
    \sum_{j=0}^{k}\Big(1-\frac{\lambda}{d}\Big)^j\leq \frac{d}{\lambda}.
\end{align*}
\end{proof}
In order to prove the upper bound in Theorem \ref{MAINTHEOREM}, we will need a more refined control on the probability of excursions of the Gibbs sampler away from the main diagonal. 
\begin{lemma}\label{largedeviations}
Let $C$ be a network on $d$ vertices satisfying Assumption \ref{assumption}, and $A^2\geq \frac{1}{2\beta}$. For any $\delta>0$, let $\eta>0$ be such that 
\begin{equation}\label{conditionforeta}
    \frac{\lambda^2\beta^2\delta^2}{2d}-2dA^2\eta^2\geq \eta.
\end{equation}
Let $\bold p_0\in \mathcal D$ and, for $k\in\mathbb N$, let $\bold p(k)\sim K_{A,C}^{*k}(\bold p_0)$. Then we have the bound
\begin{align*}
\mathbb P\Big(\max_{i,j\in [d], t\in [k]}|p_i(t)-\hat p_j(t)|)\geq \delta\Big)\leq 13ke^{-\eta A^2}.
\end{align*}
In particular, if $A\geq A^*(\delta,\beta,\lambda)$, 
\begin{equation}\label{sharpoffidiagonal}
\mathbb P\Big(\max_{i,j\in [d], t\in [k]}|p_i(t)-\hat p_j(t)|)\geq \delta\Big)\leq 13ke^{-\frac{\lambda\beta\delta}{2d}A}.
\end{equation}
\end{lemma}
\begin{proof}
By convexity, 
\begin{align*}
\max\limits_{i,j\in [d]}|p_i-\hat p_j|\leq \max\limits_{i,j\in [d]}|p_i-p_j|.
\end{align*}Combining with \eqref{inner} and \eqref{spectralbounds}, if $\max\limits_{i,j\in [d]}|p_i(t)-\hat p_j(t)|\geq \delta$ occurs at some $t$ then we have the chain of inequalities
\begin{align*}
    \beta^2\delta^2\leq c_ic_j|p_i(t)-p_j(t)|^2\leq 2\langle \bold p(t)-\overline p(t)\bold 1,\bold p(t)-\overline p(t)\bold 1\rangle\leq \frac{2}{\lambda}\langle \bold p(t),\Delta \bold p(t)\rangle, 
\end{align*}
which leads to the inclusion of events
\begin{align*}
   \Big\{\max_{i,j\in [d], t\in [k]}|p_i(t)-\hat p_j(t)|)\geq \delta \Big\}\subset \Big\{\max_{t\in [k]}\langle \Delta \bold p(t), \bold p(t)\rangle \geq \frac{\lambda\beta^2 \delta^2}{2}\Big\}=:B_k
\end{align*}
In what follows, $I=I_t$ denotes the coordinate selected by the Gibbs sampler at step $t$. For $\eta>0$ as in \eqref{conditionforeta} and $m:=dA^2\eta$, define the events
\begin{align*}
    &B_k(\eta):=\Big\{\langle \Delta \bold p(t),\bold p(t)\rangle-\langle \Delta \bold p(t-1),\bold p(t-1)\rangle\geq \eta \text{ for some }t\in [k] \Big\},\\ 
    &B_k(m):=\Big\{\exists T= \{h,\ldots, h+m-1\}\subset [k], \Big |I_{t-1}\neq \argmax_{i\in [d]} \{c_i|\Delta p_i(t-1)|^2\}, t\in T\Big\},\\
    &B_k(m,\eta):=B_k\setminus \Big(B_k(\eta)\cup B_k(m)\Big).
\end{align*}
We use the notation $\bold q=\bold p(t-1)$, $\bold q'=\bold p(t)$. As we have shown in Lemma \ref{lemmaexcursion},
\begin{equation}\label{singleclusteringupdate}
    \langle \Delta \bold q',\bold q'\rangle-\langle \Delta \bold q,\bold q\rangle=-c_{I}|(\Delta q)_{I}|^2+c_{I}|\epsilon_{I}(\sigma^2_{I}, \hat q_{I})|^2.
\end{equation}
Therefore, if $Z$ denotes a standard normal random variable, a union bounds yields
\begin{align*}
    \mathbb P(B_k(\eta))\leq k\frac{\mathbb P(Z^2\geq 2A^2\eta)}{\mathbb P(Z\in [-2c_IA^2\hat q_I,2c_IA^2(1-\hat q_I)])}\leq 6ke^{- \eta A^2},
\end{align*}
where the last inequality follows from $A^2c_I\geq A^2\beta\geq 1$ and the log-concavity of the normal distribution, which together gives $\mathbb P(Z\in [-\hat q_I,1-\hat q_I])\geq \mathbb P(Z\in [0,1])\geq \frac{1}{3}$.\\ \\
As for $B_k(m)$, a union bound, $\ln(1+x)\leq x$ $(x\in\mathbb R)$ and the definition of $m$ gives
\begin{align*}
    \mathbb P(B_k(m))\leq k\Big(1-\frac{1}{d}\Big)^{m}\leq ke^{-\frac{m}{d}}\leq ke^{-\eta A^2}
\end{align*}
In the case of $B_k(m,\eta)$, consider the time $t^*$ at which $\langle \Delta\bold p(t^*),\bold p(t^*)\rangle$ reaches the target $\lambda\beta^2\delta^2/2$. Going backwards in time for $m$ steps, we cannot be further than $m\eta$ from the target. Also, the coordinate maximizing $c_i|\Delta p_i(t)|^2$ was selected at least once in the time interval $[t^*-m,t^*-1]$. More precisely, the occurrence of $B_k(m,\eta)$ implies that
there is $t\in[t^*-m,t^*-1]$ where, using again the notation $\bold q=\bold p(t-1), \bold q'=\bold p(t)$, 
\begin{align*}
\langle \Delta \bold q,\bold q \rangle \in \Big(\frac{\lambda\beta^2 \delta^2}{2}&-m\eta, \frac{\lambda\beta^2 \delta^2}{2}\Big), \quad I=\argmax_{i\in [d]}c_i|\Delta q_i|^2, \quad \langle \Delta \bold q',\bold q'\rangle \geq \frac{\lambda\beta^2 \delta^2}{2}-m\eta.
\end{align*}
Combining with \eqref{singleclusteringupdate} and Lemma \ref{cutefact}, the event $B_k(m,\eta)$ implies 
\begin{align*}
    c_{I}|\epsilon_{I}(\sigma^2_{I}, \hat q_{I})|^2&\geq -m\eta+c_{I}|\Delta q_{I}|^2\\&\geq -m\eta+\frac{1}{d}\langle \Delta\bold q,\Delta\bold q\rangle
    \\&\geq-m\eta +\frac{\lambda}{d}\langle\Delta\bold q,\bold q\rangle \\&\geq \frac{\lambda^2\beta^2\delta^2}{2d}-\Big(\frac{\lambda}{d}+1\Big)dA^2\eta^2\\& \geq \eta
\end{align*}
where we also used $\frac{\lambda}{d}\leq 1$, the definition of $m$ and \eqref{conditionforeta}. 
Therefore, using a union bound and $A^2\beta\geq 1$, 
\begin{align*}
    \mathbb P(B_k(m,\eta))\leq 6ke^{-\eta A^2}.
\end{align*}
Collecting all the terms and using a union bound over all coordinates, we obtain the first claim. As for the second, notice that for $A\geq A^*(\delta,\beta,\lambda)$, the choice of
\begin{align*}
    \eta=\frac{\lambda\beta\delta}{2 dA}.
\end{align*}
satisfies \eqref{conditionforeta}. Here we used that, since $\beta\leq\frac{1}{d}$ for any network $C$, we can define $A^*(\delta,\beta,\lambda,d)=A^*(\delta,\beta,\lambda)$.
\end{proof}

\section{Diffusive behavior}\label{diffusivesection}
This section is devoted to an in-depth understanding of the long-term movement of the walk both inside the hypercube and also when approaching the boundary, which will serve as the core of the proof of Theorem \ref{MAINTHEOREM}. The relevant quantity to follow is a network-adapted orthogonal projection onto the diagonal of the hypercube. For a point $\bold p\in [0,1]^d$, we define its (weighted) barycenter by
\begin{equation}\label{barycenterdef}
    \overline p:=\langle \bold p,\bold 1\rangle.
\end{equation}
Let $\bold p_0\in [0,1]^d$, and $\bold p(k)\sim K^{*k}_{A,C}(\bold p_0)$. The main intuition behind the $A^2$ scaling appearing in Theorem \ref{MAINTHEOREM} is twofold: the evolution of $\overline p(k)$, the barycenter of $\bold p(k)$, resembles that of a simple random walk, and by the time $\overline p(k)$ gets close to stationarity, all coordinates of $\bold p(k)$ are already clustered together. The first claim is misleading -- the evolution of $\overline p(k)$ involves jumps and the knowledge of the whole chain -- but we will show how a certain diffusive-like behavior emerges when $A$ is sufficiently large. As for the second claim, since we work in the Wasserstein metric \eqref{metricwecare}, Lemma \ref{largedeviations} will suffice.
\subsection{Anticoncentration vs. concentration}
Before analyzing the walker's barycenter, we need to properly understand the stationary distribution $\pi_{A,C}$, for $C$ fixed and $A$ large enough. Under $\pi_{A,C}$ we expect the barycenter to be roughly uniformly distributed on the unit interval. We are able to obtain the following quantitative statement, which gives sharp anti-concentration bounds. In the following, $\overline \pi_{A,C}$ will denote the distribution of $\overline p$, where $\bold p\sim \pi_{A,C}$.
\begin{lemma}\label{spreadforstat}
Let $C$ be a network on $d$ vertices satisfying Assumption \ref{assumption}. Let $\delta>0$ and let $A$ be such that
\begin{equation}\label{howlargeA}
    A^2\geq \frac{\sqrt e d}{2\lambda \beta\delta^3}.
\end{equation}
Then, for every $s\in (\delta,1-\delta)$, one has 
\begin{equation}\label{LOWERBOUNDSTAT}
    \overline{\pi}_{A,C}([s,1-s])\leq \frac{1-2s}{(1-2\delta)^2}.
\end{equation}
In particular, for any $\delta>0$ and all $A\ge A^*(\delta,\beta,\lambda)$, if $\overline p\sim \overline{\pi}_{A,C}$ then 
\begin{equation}\label{LOWERBOUND_STAT}
    \mathbb E\Big(\Big|\overline p-\frac{1}{2}\Big|\Big)\geq \frac{1}{4}-\delta.
\end{equation}
\end{lemma}
\begin{proof}
With a bit of abuse of notation, let $\overline{\pi}_{A,C}(t)$ denote also the density of $\overline{\pi}_{A,C}$ at $t\in [0,1]$. Write $\bold p=\overline p\bold 1+\bold q$ with $\bold q\in \mathcal D^{\perp}$. By means of \eqref{hereislap} we can write
\begin{equation}\label{tindependent}
    \overline{\pi}_{A,C}(t)\propto \int_{\bold q\in [-t\bold 1, (1-t)\bold 1]\cap \mathcal D^{\perp}}e^{-A^2\langle \bold q,\Delta\bold q\rangle }d\bold q.
\end{equation}
For any $t\in [0,1]$, the right side in \eqref{tindependent} is upper bounded by
\begin{align*}
    \int_{\mathcal D^{\perp}} e^{-A^2\langle \bold q,\Delta\bold q\rangle }d\bold q=\frac{Z(C)}{A^{d-1}}, \quad Z(C):=\int_{\mathcal D^{\perp}} e^{-\langle \bold q,\Delta\bold q\rangle }d\bold q<+\infty,
\end{align*}
where the $(d-1)$-scaling appears since we are integrating on a $(d-1)$ subspace. 
On the other hand, for $t\in (\delta ,1-\delta)$ we have
\begin{align*}
    \frac{Z(C)}{A^{d-1}}-\overline{\pi}_{A,C}(t)&\leq \frac{1}{A^{d-1}}\int_{\mathcal D^{\perp}\cap [-A\delta \bold 1,A\delta\bold 1]^{c}}e^{- \langle\bold q,\Delta \bold q \rangle }d\bold q\\&\leq \frac{1}{A^{d-1}}\frac{1}{\lambda \beta A^{2}\delta^2}\int_{\mathcal D^{\perp}}e^{-\langle \bold q,\Delta \bold q\rangle}\langle \bold q,\Delta\bold q\rangle d\bold q\\&
    \leq \frac{Z(C)}{A^{d-1}}\frac{\sqrt e d}{\lambda \beta A^2\delta^2}.
\end{align*}
Here, in the second last inequality we used that, owing to \eqref{inner} and \eqref{spectralbounds}, 
\begin{align*}
    \langle \bold q, \Delta\bold q\rangle \geq \lambda \langle\bold q,\bold q\rangle\geq \lambda\beta \sum\limits_{i\in [d]}q_i^2 \geq \lambda \beta A^2\delta^2,\quad \bold q\in [-A\delta\bold 1, A\delta\bold 1]^{c},
\end{align*}
while the last inequality follows from $x\leq e^{x}$ ($x\in \mathbb R$), a change of variables and 
\begin{align*}
    \Big(1-\frac{1}{d}\Big)^{d-1}\geq \frac{1}{e}, \quad d>1.
\end{align*}
Altogether, for $A$ as in \eqref{howlargeA}, we obtain 
\begin{align*}
\sup_{t,t'\in [\delta,1-\delta]}\frac{\overline {\pi}_{A,C}(t)}{\overline {\pi}_{A.C}(t')}\leq \frac{1}{1-2\delta}.  
\end{align*}
Therefore, for $s\in [\delta,1-\delta]$, 
\begin{align*}
\overline{\pi}_{A,C}([s,1-s])\leq \frac{\overline{\pi}_{A,C}([s,1-s])}{\overline{\pi}_{A,C}([\delta,1-\delta])}\leq \frac{1-2s}{(1-2\delta)^2},
\end{align*}
which concludes the proof of \eqref{LOWERBOUNDSTAT}. As for \eqref{LOWERBOUND_STAT}, fix $\delta>0$, and apply \eqref{LOWERBOUNDSTAT} with $\delta'=\delta'(\delta)$ such that 
\begin{align*}
    1-\frac{1}{(1-2\delta')^2}\leq \frac{\delta}{2}, \quad \delta'\leq \frac{\delta}{2},
\end{align*}
and choose $A$ large accordingly. Then
\begin{align*}
\mathbb E\Big(\Big|\overline p-\frac{1}{2}\Big|\Big)&=\int_{0}^{\frac{1}{2}}\Big(1-\overline \pi_{A,C}([s,1-s])\Big)ds\\&
\geq \int_0^{\frac{1}{2}-\delta'}(1-(1-2s))ds-\frac{\delta}{2} \\&\geq \int_0^{\frac{1}{2}}2sds-\delta\\&=\frac{1}{4}-\delta.
\end{align*}
\end{proof}
Consider now the situation where $\bold p_0=\frac{1}{2}\bold 1$. After few steps of the Gibbs sampler, coordinates will not have enough time to travel away from the middle. This is done by a concentration argument.
\begin{lemma}\label{barycenterdidnottravel}
Let $C$ be a network on $d$ vertices satisfying Assumption \ref{assumption}, and let $A>0$, $k\in \mathbb N$. Let $\bold p_0=\frac{1}{2}\bold 1$, $\bold p(k)\sim K_{A,C}^{*k}(\bold p_0)$ with barycenter $\overline p(k)$. Then, for every $\delta>0$, 
\begin{align*}
\mathbb E\Big[\Big(\overline p(k)-\frac{1}{2}\Big)^2\Big]\leq \frac{27k\gamma}{\lambda A^2}.
\end{align*}
In particular, for $m=\frac{\delta^2}{27}$ and $k\leq m\frac{\lambda}{\gamma}A^2$, 
\begin{equation}\label{LOWERBOUNDGIBBS}
    \mathbb E\Big[\Big|\overline p(k)-\frac{1}{2}\Big|\Big]\leq \delta.
\end{equation}
\end{lemma}
\begin{proof}
Since the walk is invariant under reflection of the hypercube with respect to $\bold p_0$, we get immediately $\mathbb E(\overline p(k))\equiv \frac{1}{2}$ for all $k$. For the sake of convenience, in the following we write $\bold q:=\bold p(k-1), \bold q':=\bold p(k)$ (and, correspondingly, $\overline q, \overline q'$ for their barycenters), and drop the dependence on $k$ for the moment. From \eqref{totalstep} and \eqref{barycenterdef}, conditional on index $i$ being chosen, we get
\begin{align*}
    \Big(\overline q'-\frac{1}{2}\Big)^2-\Big(\overline q-\frac{1}{2}\Big)^2=&2\Big(\overline q-\frac{1}{2}\Big)\langle -\Delta(\bold q)+\boldsymbol{\epsilon}(\Sigma, \hat{\bold q}),\Pi_i\bold 1\rangle\\&+\Big(\langle -\Delta \bold q+\boldsymbol{\epsilon}(\Sigma, \hat{\bold q}), \Pi_i \bold 1\rangle\Big)^2.
\end{align*}
Notice the identities
\begin{align*}
  \sum_{i\in [d]}\Pi_i\bold v=\bold v,\quad \sum\limits_{i\in [d]}|\langle \bold v,\Pi_i\bold 1\rangle|^2=\langle D\bold v,\bold v\rangle,\quad \bold v\in\mathbb R^d,  
\end{align*} 
and the fact $\Delta \bold q\in\mathcal D^{\perp}$.
 We now take expectation on both sides and use the tower property to exploit the above (first conditioning with respect to $\bold q$, then with respect to $\boldsymbol \epsilon$).
The right side then becomes
\begin{align*}
     \frac{2}{d}\mathbb E\Big(\Big(\overline q-\frac{1}{2}\Big)\langle \mathbb E\Big(\boldsymbol{\epsilon}(\Sigma, \hat{\bold q})\Big|\bold q\Big),\bold 1\rangle\Big)+\frac{1}{d}\mathbb E\Big(\langle D\Big(\Delta \bold q+\boldsymbol{\epsilon}(\Sigma, \hat{\bold q})\Big),\Big(\Delta\bold q+\boldsymbol{\epsilon}(\Sigma, \hat{\bold q})\Big)\rangle\Big).
\end{align*}
Since $\Delta\bold q\in \mathcal D^{\perp}$, $\overline q-\frac{1}{2}=\langle \bold q-\frac{1}{2}\bold 1,\bold 1\rangle=\langle \bold{\hat q}-\frac{1}{2}\bold 1,\bold 1\rangle$. Therefore, if we expand the inner products in the first summand and apply Cauchy-Schwartz on the second one,
\begin{align*}
  \mathbb E\Big[\Big(\overline q'-\frac{1}{2}\Big)^2-\Big(\overline q-\frac{1}{2}\Big)^2\Big]\leq &\frac{2}{d}\sum_{i,j\in [d]}c_ic_j\mathbb E\Big(\Big(\hat q_i-\frac{1}{2}\Big)\mathbb E(\epsilon_j(\sigma_j^2,\hat q_j)|\bold q)\Big)\\&+\frac{2}{d}\mathbb E(\langle D\Delta\bold q,\Delta\bold q \rangle)\\&+\frac{2}{d}\mathbb E(\langle D \boldsymbol{\epsilon}(\Sigma, \hat{\bold q}),\boldsymbol{\epsilon}(\Sigma, \hat{\bold q}) \rangle)\\&=:E_1+E_2+E_3. 
\end{align*}
We start bounding $E_1$. From \eqref{explicitmean}, positive contribution only comes from terms where $\hat q_i-\frac{1}{2}$ and $\hat q_j-\frac{1}{2}$ have opposite signs, in which case $|\hat q_i-\frac{1}{2}|\leq |\hat q_i-\hat q_j|$. Thanks to this observation,
\begin{align*}
    E_1&\leq \frac{2}{d}\sum_{i,j\in [d]}c_ic_j\mathbb E\Big(|(\hat q_i-\hat q_j)\mathbb E(\epsilon_j(\sigma_j^2,\hat q_j)|\hat q_j)|\Big)\\
    &\leq \frac{2}{d}\sum_{i,j\in [d]}c_ic_j\mathbb E|\hat q_i-\hat q_j|\frac{\sqrt 2}{\sqrt{c_j}A}\\
    & \leq \frac{2\sqrt 2}{dA}\Big(\sum_{i,j\in [d]}c_ic_j\mathbb E|\hat q_i-\hat q_j|^2\Big)^{\frac{1}{2}}\Big(\sum_{i,j\in [d]}c_i\Big)^{\frac{1}{2}}\\
    &= \frac{4}{A\sqrt{d}}\Big(\mathbb E(\langle \hat{\bold q}-\overline q\bold 1,\hat{\bold{q}}-\overline q\bold 1\rangle)\Big)^{\frac{1}{2}}\,
\end{align*}
where we used \eqref{explicitmean}, the Cauchy-Schwartz inequality and the identity 
\begin{align*}
    \sum_{i,j\in [d]}c_ic_j|\hat q_i-\hat q_j|^2=2\langle \hat{\bold q}-\overline q\bold 1,\hat{\bold q}-\overline q\bold 1 \rangle.
\end{align*}
Finally, using \eqref{lambdastuff}, \eqref{spectralbounds} and \eqref{excursionbound} we obtain
\begin{align*}
    E_1&\leq \frac{4\gamma}{A\sqrt d}\Big(\mathbb E(\langle\bold q-\overline q\bold 1,\bold q-\overline q\bold 1\rangle)\Big)^{\frac{1}{2}}\\&\leq 
    \frac{4\gamma}{A\sqrt d\sqrt{\lambda}}\Big(\mathbb E(\langle \bold q,\Delta\bold q\rangle)\Big)^{\frac{1}{2}}\\&\leq \frac{2\sqrt{10}\gamma}{\lambda A^2}
\end{align*}
As for $E_2$, by means of \eqref{spectralbounds}, \eqref{excursionbound} and Lemma \ref{cutefact} we can bound 
\begin{align*}
    E_2&\leq \frac{4}{d}\max_{i\in [d]}c_i\mathbb E(\langle \Delta\bold q,\bold q\rangle )\\&\leq \frac{10\gamma}{\lambda A^2}.
\end{align*}
Finally, for $E_3$ we can use \eqref{varup}, Lemma \ref{cutefact} and $\lambda\leq 2$ to obtain 
\begin{align*}
    E_3&\leq \frac{2\max_{i\in [d]}c_i}{d}\sum_{i\in [d]}c_i\frac{5}{2c_iA^2}\\&\leq \frac{10\gamma}{\lambda A^2}.
\end{align*}
Combining the bounds for $E_1, E_2, E_3$ and using a telescoping sum, we obtain 
\begin{align*}
    \mathbb E\Big[\Big(\overline p(k)-\frac{1}{2}\Big)^2\Big]\leq \frac{27k\gamma}{\lambda A^2}.
\end{align*}
The last statement then follows at once from Jensen inequality.
\end{proof}

The claim of Lemma \ref{barycenterdidnottravel} is in line with simulation results as shown below in Figure \ref{fig:variancegrowth} for the complete graph with uniform weights, for which $\frac{\lambda}{\gamma}A^2=dA^2$. Observe that the graph stabilizes around $\approx\frac{1}{12}$ when $A$ large, corresponding to the variance of the uniform distribution on $[0,1]$ (a good approximation of $\overline \pi_{A,C}$ for $A$ large because of Lemma \ref{spreadforstat}). The scaling of the variance -- quadratic in $A$, linear in $d$ -- agrees with our bound.
\begin{figure}[h]
  \centering
  \subfloat[$d=4,A=150$]{
    \includegraphics[width=0.45\textwidth]{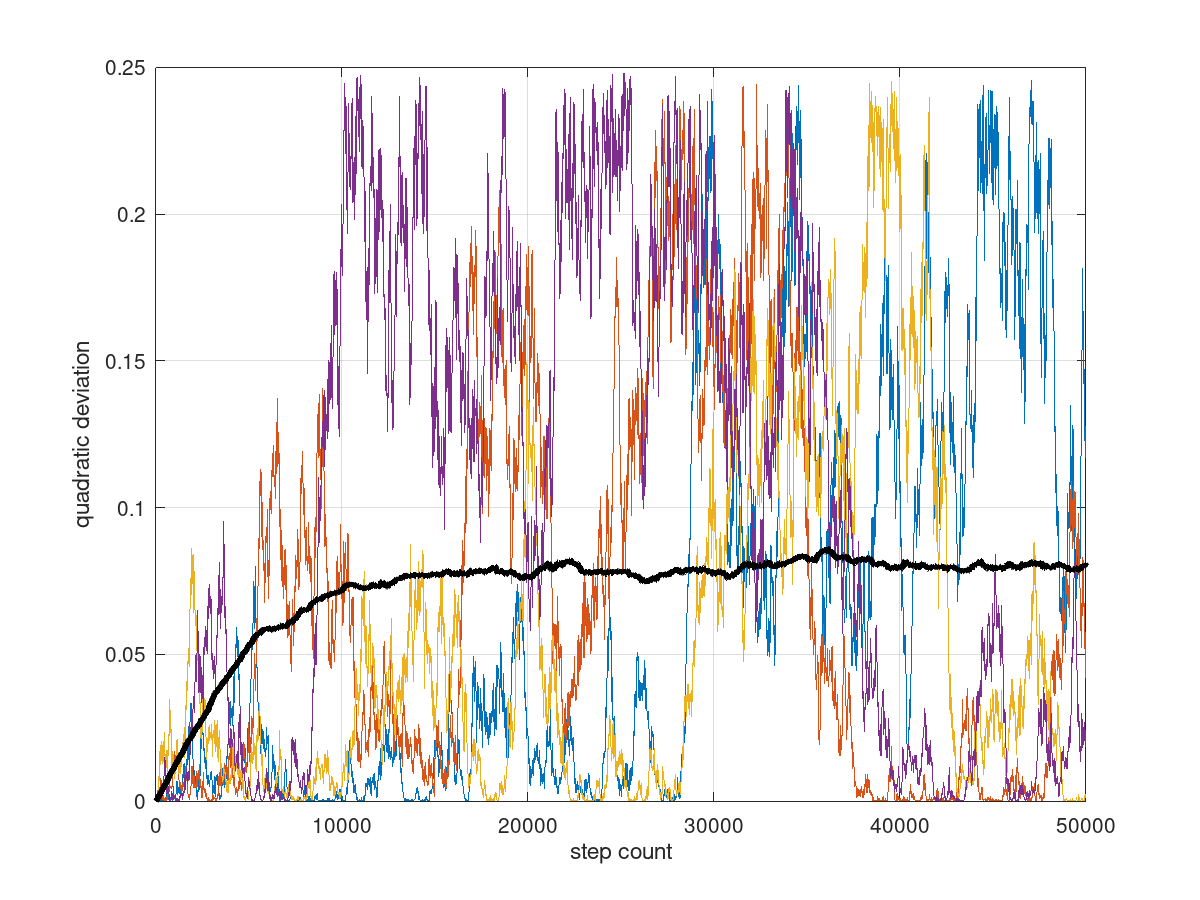}
  }
  \subfloat[$d=4,A=300$ ]{
    \includegraphics[width=0.45\textwidth]{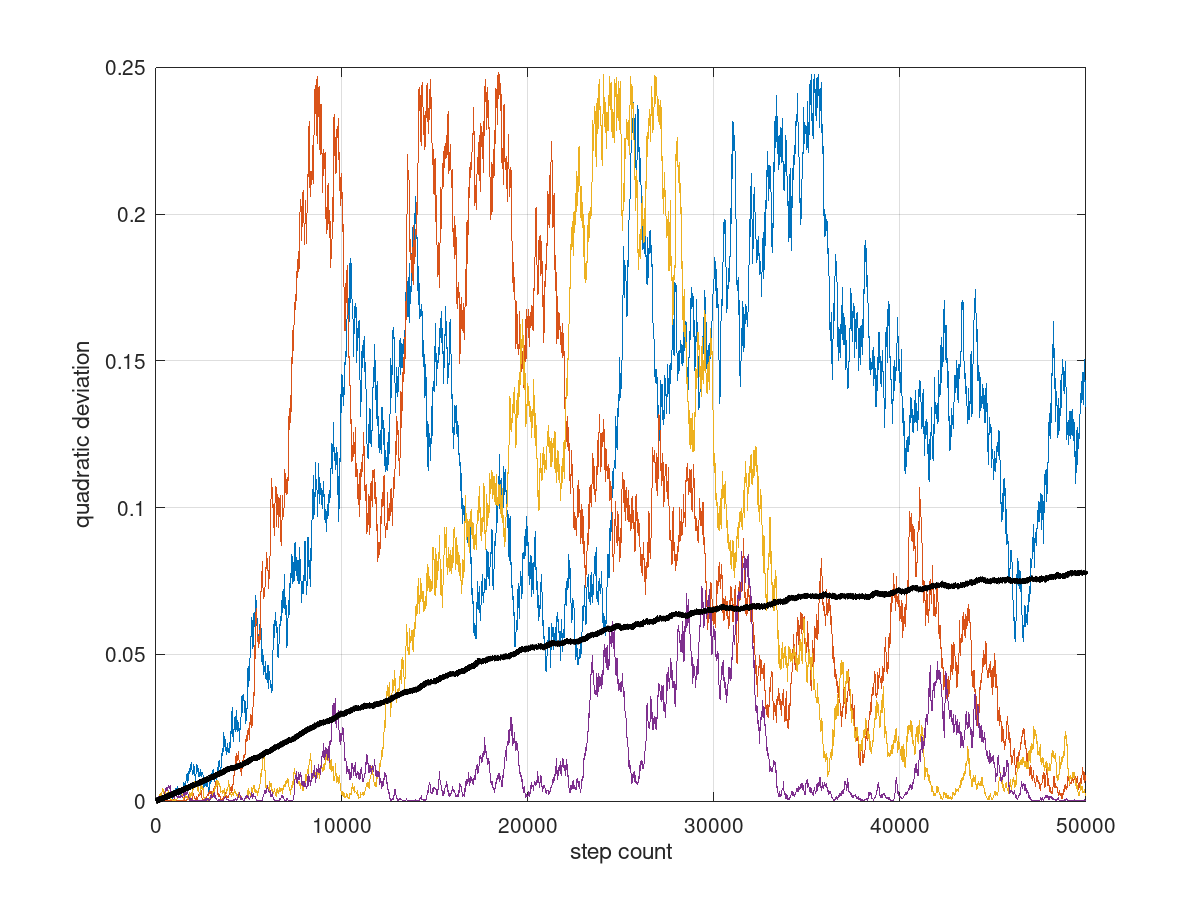}
  }
  
  \subfloat[$d=8,A=150$ ]{
    \includegraphics[width=0.45\textwidth]{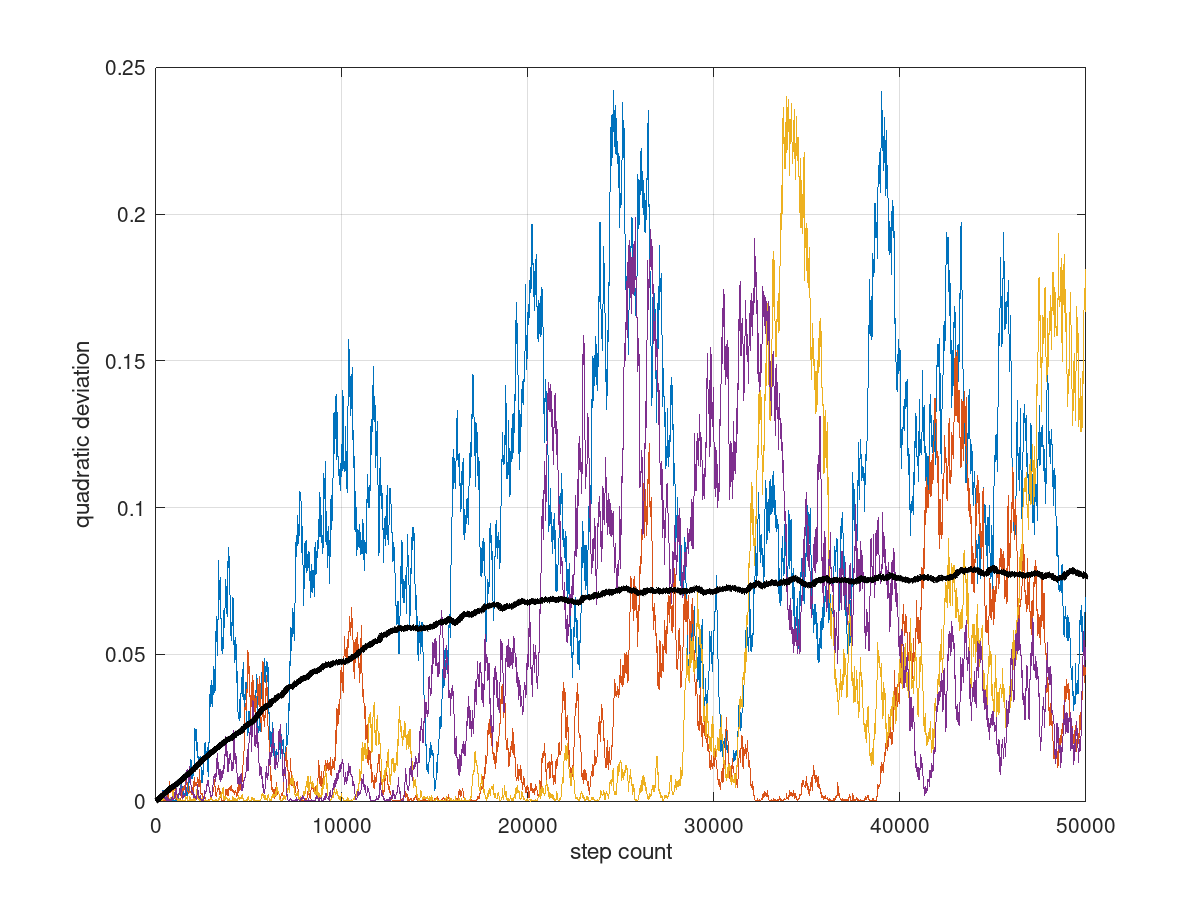}
  }
  \subfloat[$d=8,A=300$ ]{
    \includegraphics[width=0.45\textwidth]{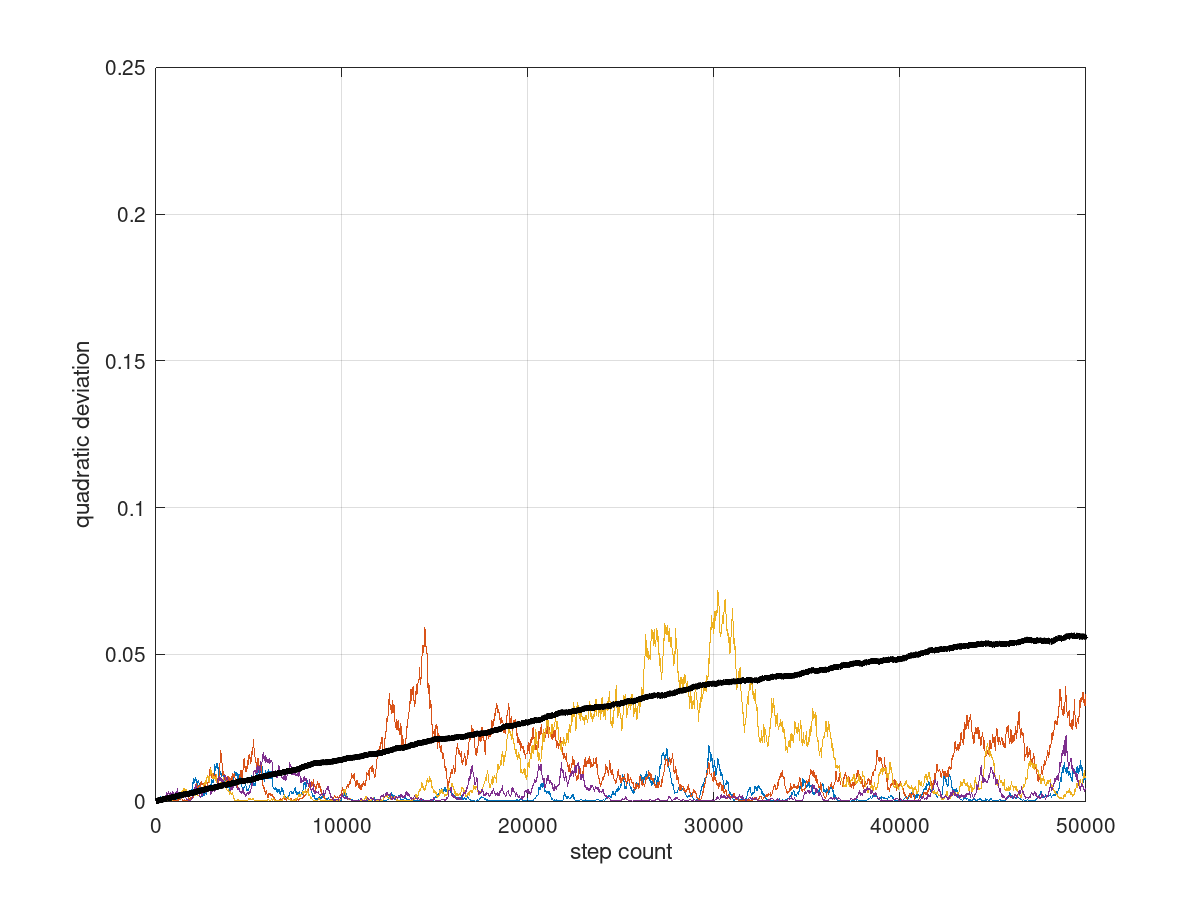}
  }
  \caption{A few sample paths of the evolution of the squared distance of the barycenter from $\frac 1 2$ together with an empirical mean of 1000 runs for each, for different values of $d$ and $A$, considering the complete graph with uniform weights.}
  \label{fig:variancegrowth}
\end{figure}

\subsection{A stopping time for the boundary}
In this subsection, let $\bold p_0:=\bold 0$ and let $\bold p(k)\sim K^{*k}_{A,C}(\bold p_0)$. For $\delta\in (0,1)$, consider the stopping time
\begin{equation}
    T_{\delta}=\min_{k\in \mathbb N}\Big\{ \max_{i\in [d]}(1-\hat p_i(k))\leq \delta\Big\}.
\end{equation}
Since the chain is ergodic and
\begin{align*}
\pi_{A,C}\Big(\Big\{\max_{i\in [d]}(1-\hat p_i)\leq \delta\Big\}\Big)>0,
\end{align*}
it is easy to show that $T_{\delta}$ has finite expectation. For $H>0$, consider the random variable 
\begin{equation}
    w_H:=\overline p^2(T_{\delta})-HT_{\delta}.
\end{equation}
The intuition is that if $\overline p$ were a simple random walk, $w_H$ would be a sub-martingale for $H$ small enough.  Since, by time $T_{\delta}$, there should not be a significant negative drift, the hope is that for $H=O\Big(A^{-2}\Big)$ the expected value of $w_H$ will be positive, which can then be turned into an upper bound for $\mathbb E(T_{\delta})$. This is done in the following lemma.
\begin{lemma}\label{stoppingtime}
Let $C$ be a network on $d$ vertices satisfying Assumption \ref{assumption}. For $\delta\in (0,1)$, let $A\geq \frac{1}{2\beta}$ and large enough so that
\begin{equation}\label{choiceforH}
    H:=\frac{\rho}{2dA^2}-\frac{2\sqrt 2}{A\sqrt{d}}e^{-\delta^2\beta A^2}>0,
\end{equation}
where $\rho$ is given by Lemma \ref{truncatedstuff}.\\
Then, $\mathbb E(w_{H})>0$. In particular, if $A\geq A^*(\delta,\beta)$, then
\begin{equation}\label{boundonstoppingtime}
    \mathbb E(T_{\delta})\leq \frac{4dA^2}{\rho}.
\end{equation}
\end{lemma}
\begin{proof}
Since $T_{\delta}$ is finite almost surely, summation by parts gives 
\begin{align*}
    \mathbb E(w_H)&=\sum_{j=1}^{+\infty}\mathbb E\Big(1_{\{T_{\delta}=j\}}(\overline p^2(j)-Hj)\Big)\\&=\sum_{j=1}^{+\infty}\mathbb E\Big(1_{\{T_{\delta}>j-1\}}\Big(\overline p^2(j)-\overline p^2(j-1)-H\Big)\Big).
\end{align*}
Since $\{T_{\delta}>j-1\}$ is measurable with respect to the $\sigma$-algebra generated by $\bold p(j-1)$, it suffices to show that on the relevant event
\begin{equation}\label{equivalentstatement}
    \mathbb E(\overline p^2(j)-\overline p^2(j-1)|\bold p(j-1))\geq H.
\end{equation}
We drop the dependence on $j$, and denote $\bold p(j)=\bold q', \bold p(j-1)=\bold q$. Proceeding as in the proof of Lemma \ref{barycenterdidnottravel}, using \eqref{barycenterdef} and \eqref{totalstep} we obtain
\begin{align*}
 \mathbb E(\overline q'^2-\overline q^2|\bold q)&=\frac{2\overline q}{d}\mathbb\langle \mathbb E(\boldsymbol{\epsilon}(\Sigma, \hat{\bold q})|\bold q),\bold 1\rangle+\frac{1}{d}\mathbb E (\langle D(\Delta \bold q+\boldsymbol{\epsilon}(\Sigma, \hat{\bold q})), \Delta \bold q+\boldsymbol{\epsilon}(\Sigma, \hat{\bold q})\rangle) \\&=:E_1+E_2.
\end{align*}
On the event $\{T_{\delta}>j-1\}$, we have $\bold {\hat q}\leq (1-\delta)\bold 1$. If we expand the inner product appearing in $E_1$, because of \eqref{explicitmean} we only get negative contributions if $\hat q_i>\frac{1}{2}$. Therefore, we can lower bound
\begin{align*}
    E_1&\geq -\frac{2}{d}\sum\limits_{i\in d}c_i\frac{2}{\sqrt 2A\sqrt c_i}e^{-\delta^2c_iA^2}\\&\geq 
    -\frac{2\sqrt 2}{A\sqrt d}e^{-\delta^2\beta A^2},
\end{align*}
where we used Cauchy-Schwartz in the last inequality.
As for the second term, the definition of variance and \eqref{vardown} gives, for every $A\geq \frac{1}{2\beta}$,
\begin{align*}
    E_2&= \frac{1}{d}\sum_{i\in [d]}c_i^2\mathbb E\Big[\Big(\epsilon_i(\sigma_i^2, \hat q_i)-\Delta q_i\Big)^2\Big|\bold q\Big]\\&\geq \frac{\rho}{d}\sum_{i\in [d]}c_i^2\frac{1}{2c_iA^2}\\&\geq \frac{\rho}{2dA^2}.
\end{align*}
Therefore, using the definition of $H$ in \eqref{choiceforH} we obtain
\begin{align*}
    \mathbb E(\overline p^2(j)-\overline p^2(j-1)|\bold p(j-1))\geq E_1+E_2\geq H,
\end{align*}
which proves \eqref{equivalentstatement}. The last statement follows at once from definition of $w_H$ and the constraint $\overline p\in [0,1]$.
\end{proof}

\bigskip

Figure \ref{fig:hitmaxrange} shows the average values of $T_{0.05}$ for a range of $A$ and two distinct value of $d$. The complete graph with uniform weights is used once again for demonstration. The quadratic increase in $A$ and the linear dependence on $d$ are visible from the plots. The need for $A$ large enough is also apparent.
\begin{figure}[h]
    \centering
    \subfloat[$d=4$]{
        \includegraphics[width=0.45\textwidth]{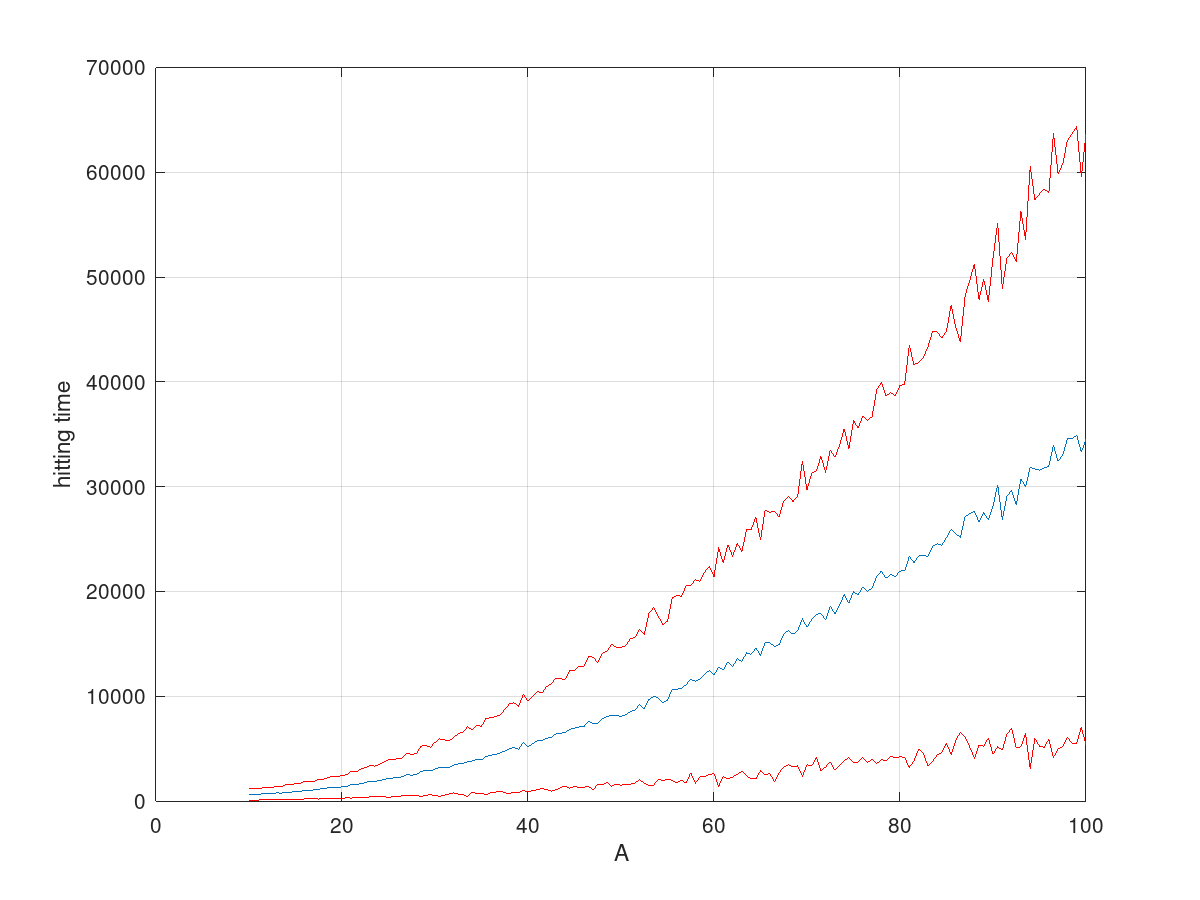}
     } 
  \subfloat[$d=8$]{
    \includegraphics[width=0.45\textwidth]{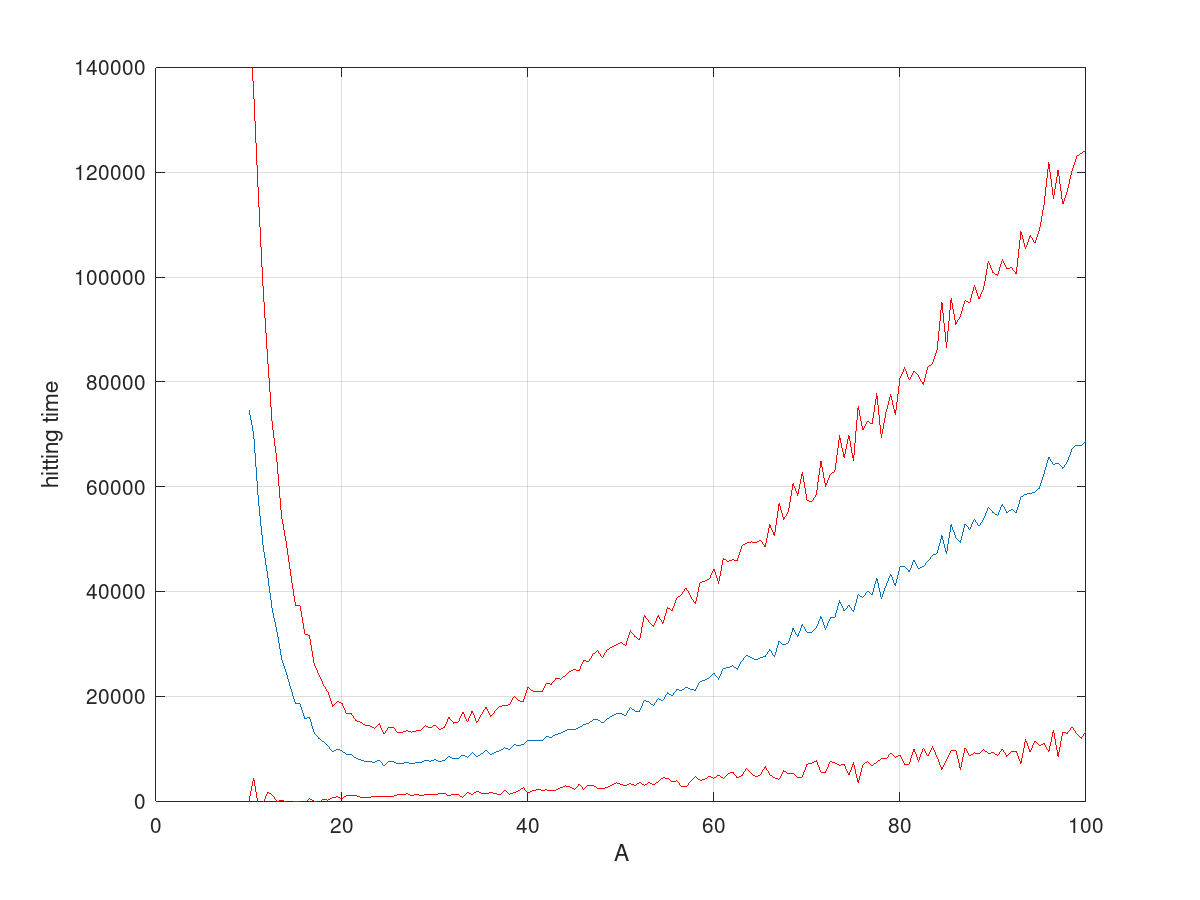}
  }
    \caption{Mean value of the hitting time of $T_{0.05}$, with over 1000 runs for each value of $A$. The stripe with a width of the empirical standard deviation is also shown.}
    \label{fig:hitmaxrange}
\end{figure}

\section{Stitching the elements}\label{sectionproof}

We first prove our result for connected networks.

\begin{proof}[Proof of Theorem \ref{MAINTHEOREM}]
For the lower bound, let $\bold p_0=\frac{1}{2}\bold 1$ and choose the optimal coupling $(\bold p(k), \bold q)$ as needed for \eqref{metricwecare}, where $\bold p(k)\sim K_{A,C}^{*k}(\bold p_0)$ and $\bold q\sim \pi_{A,C}$. Notice that $\|\bold p-\bold q\|_{\infty}\geq |\overline p(k)-\overline q|$ by convexity. For given $\delta>0$, take $A$ is large enough as in \eqref{LOWERBOUND_STAT} and $k\leq m(\delta)\frac{\lambda}{\gamma}A^2$, where $m$ is the one given by Lemma \ref{barycenterdidnottravel}. We can combine the concentration result \eqref{LOWERBOUNDGIBBS} from Lemma \ref{barycenterdidnottravel}, together with the anti-concentration result \eqref{LOWERBOUND_STAT} from Lemma \ref{spreadforstat}, and obtain
\begin{align*}
    d_{\infty}(K_{A,C}^{*k}(\bold p_0),\pi_{A,C})&=\mathbb E(\|\bold p(k)-\bold q\|_{\infty})\\&\geq \mathbb E(|\overline p(k)-\overline q|)\\&\geq \mathbb E\Big(\Big|\overline q-\frac{1}{2}\Big|\Big)-E\Big(\Big|\overline p(k)-\frac{1}{2}\Big|\Big)\\&\geq \frac{1}{4}-2\delta,
\end{align*}
from which the lower bound follows since $\delta$ is arbitrary.\\ \\
We now turn to the upper bound. For a given $\delta>0$, pick $M=M(\delta)$ to be fixed later, and $k=MdA^2$. Given $\bold p_0\in [0,1]$, consider four walkers, one starting at $\bold p_0$, one starting at $\bold 0$, one starting at $\bold 1$ and one starting at stationarity. Couple the updates according to Lemma \ref{Coup}, and denote their positions by $\bold p(k), \bold p^{(0)}(k), \bold p^{(1)}(k),$ and $\bold q(k)$ respectively. Define, for $\delta\in (0,\frac{1}{2})$, 
\begin{equation}\label{relevantstopping}
T'_{\delta}:=\inf_{k\in\mathbb N}\{\bold 1-\bold p^{(0)}(k)\leq 2\delta \bold 1\}.
\end{equation}
Because of Lemma \ref{Coup}, once $\bold p^{(0)}(k)$ is close to the opposite corner $\bold 1$, it will also be close to $\bold p^{(1)}(k)$ -- which guarantees that $\bold p(k)$ and $\bold q(k)$ are also close to each other, thanks to the order-preserving property. Moreover, Lemma \ref{Coup} ensures that the walkers will remain close at later times as well. 
\begin{align*}
    \sup_{\bold p_0\in [0,1]^d}d_{\infty}(K_{A,C}^{*k}(\bold p_0),\pi_{A,C})&\leq \mathbb E(\|\bold p(k)-\bold q(k)\|_{\infty})\\&\leq \mathbb E(\|\bold p^{(1)}(k)-\bold p^{(0)}(k)\|_{\infty})\\&\leq 2\delta+\mathbb P(T_{\delta}'>k).
\end{align*}
Notice the inclusion of the events
\begin{align*}
    \{T'_{\delta}>k\}\leq \{T_{\delta}>k\}\cup \Big\{\max\limits_{i,j\in [d], t\in [k]}|p^{(0)}_i(t)-\hat p^{(0)}_j(t)|\geq \delta\Big\}.
\end{align*}
Using a union bound, Markov inequality, the large deviation bound \eqref{sharpoffidiagonal} from Lemma \ref{largedeviations} and the bound for $\mathbb E(T_{\delta})$ in \eqref{boundonstoppingtime} from Lemma \ref{stoppingtime},  we have that for $A\geq A^*(\delta,\beta,\lambda)$  
\begin{align*}
    \sup_{\bold p\in [0,1]^d}d_{\infty}(K_{A,C}^{*k}(\bold p),\pi_{A,C})&\leq 2\delta+\frac{4dA^2}{k\rho}+13ke^{-\frac{\lambda\beta\delta A}{2d}}\\ &\leq 2\delta+\frac{4}{\rho M}+13 M dA^2e^{-\frac{\lambda\beta\delta A}{2d}}.
\end{align*}
Taking $M$ large enough, upon increasing $A^*$ we obtain
\begin{align*}
    \sup_{\bold p\in [0,1]^d}d_{\infty}(K_{A,C}^{*k}(\bold p),\pi_{A,C})\leq 4\delta,
\end{align*}
from which the upper bound follows for $k=MdA^2$ since $\delta$ is arbitrary. For $k > MdA^2$, the same argument can be run shifting the initial position from $\bold p_0$ to $\bold p(k-MdA^2)$ and exploiting the convexity of the metric.
\end{proof}

The analogous simulation on $T'_\delta$ we have seen for $T_\delta$ in Figure \ref{fig:hitmaxrange} is presented below in Figure \ref{fig:hitallrange}, for the complete graph with uniform weights, with the same values of $d$ considered there. The simulations confirm the intuition that for $A$ large enough, once the maximum is within $\delta$ of $\bold 1$, all of $\bold p$ is within $2\delta$ thanks to the clustering (with small exceptional probability). The need for a large enough $A$ is visible again.

\begin{figure}[h]
    \centering
    \subfloat[$d=4$]{
        \includegraphics[width=0.45\textwidth]{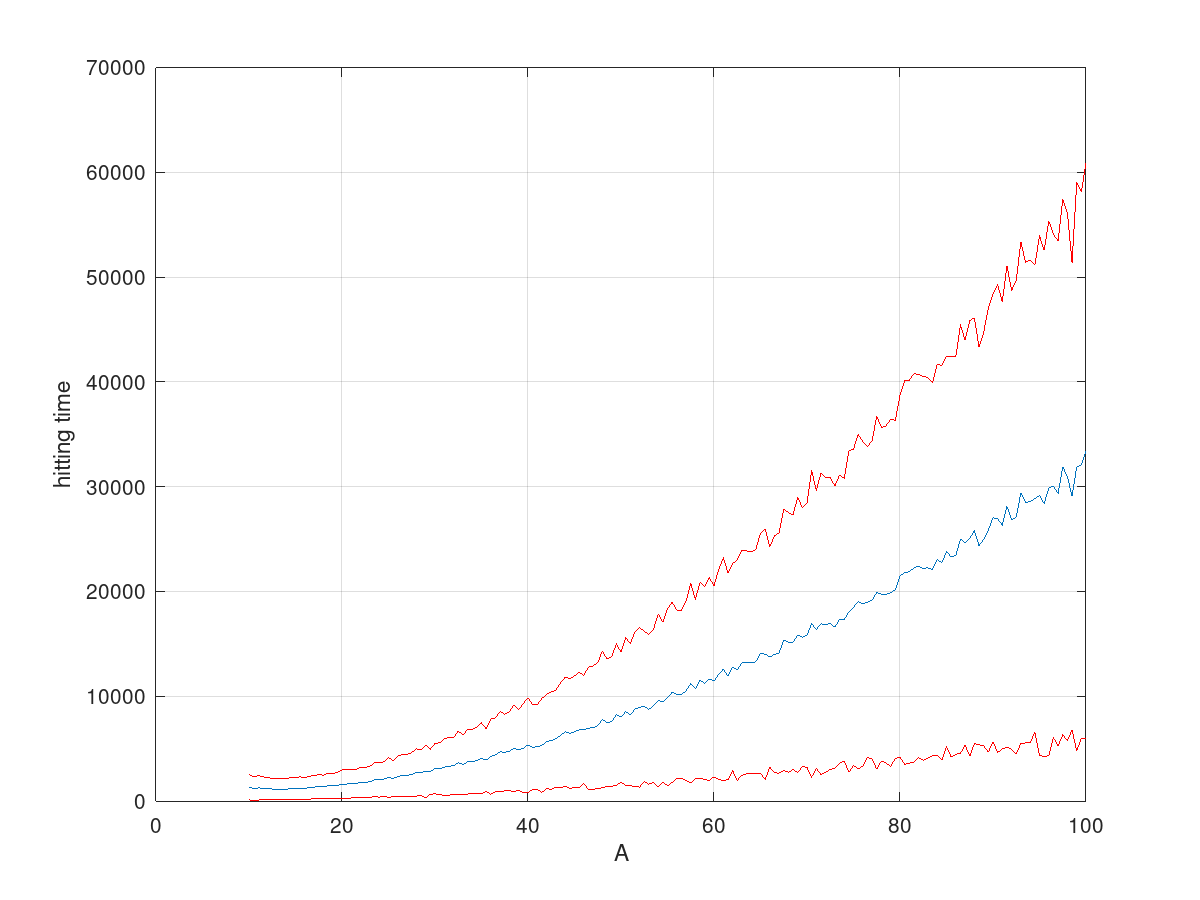}
     } 
  \subfloat[$d=8$]{
    \includegraphics[width=0.45\textwidth]{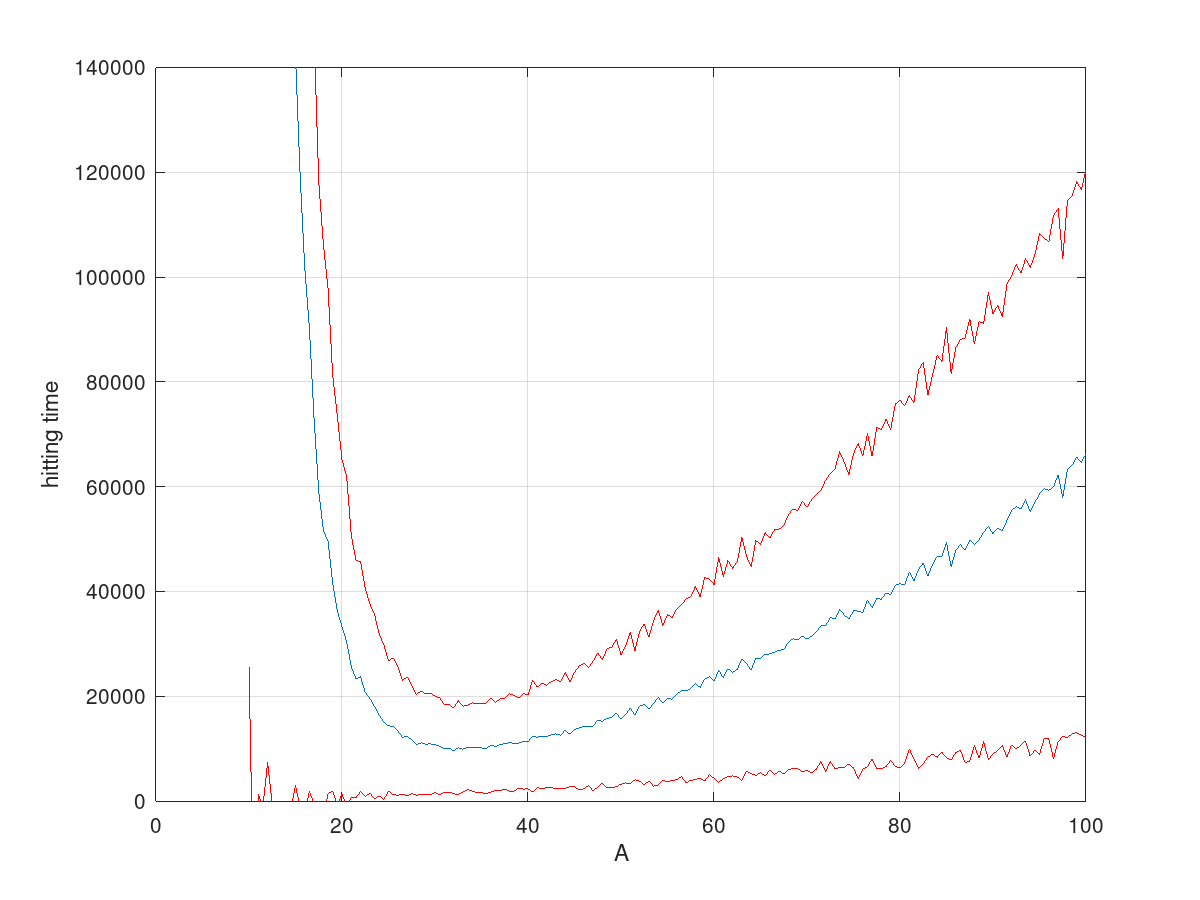}
  }
    \caption{Mean value of the hitting time of $T'_{0.05}$, with over 1000 runs for each value of $A$. The stripe with a width of the empirical standard deviation is shown.}
    \label{fig:hitallrange}
\end{figure}
We now lift the previous result to a generic non-zero network.
\begin{proof}[Proof of Theorem \ref{MAINTHEOREMDISC}.]
Let $\bold p_0=\frac{1}{2}\bold 1$ and $\delta>0$, and consider, as in the previous proof, the optimal coupling between $\bold p(k)\sim K_{A,C}^{*k}(\bold p_0)$ and $\bold q\sim \pi_{A,C}$. Call $\bold p^{(\mathcal C)}$ and $\bold q^{(\mathcal C)}$ the corresponding coordinates of $\bold p$ and $\bold q$ of any connected component -- which we called $\mathcal C$ -- with at least one edge. In order to apply Theorem \ref{MAINTHEOREM} to the network $\mathcal C$, we need to normalize its total weight $|\mathcal C|$ and absorb it in the parameter $A$. This will only change the value of $A^*$, which does not affect the current statement since $A^*$ is graph-dependent. Therefore, for suitable $A^*(\delta, C)$ and $k\leq \alpha(\mathcal C, \delta) A^2$, where
\begin{align*}
  \alpha(\mathcal C,\delta):=m(\delta)\frac{\lambda(\mathcal C)}{\gamma(\mathcal C)} |\mathcal C|,
\end{align*} 
an application of Theorem \ref{MAINTHEOREM} gives
\begin{align*}
    d_{\infty}(K_{A,C}^{*k}(\bold p_0), \pi_{A,C})&\geq \mathbb E(\|\bold p^{(\mathcal C)}(k)-\bold q^{(\mathcal C)}\|_{\infty})\\&\geq \frac{1}{4}-\delta.
\end{align*}
Maximizing over all connected component we obtain the $\alpha$ appearing in the lower bound. 
\\ 

As for the upper bound, for $\delta>0$ and $k\in\mathbb N$ define $\mathcal T_k$ to be the event where some connected component of size $n_{\mathcal C}$ is selected less than $\frac{k(1-\epsilon_{\delta})n_{\mathcal C}}{d}$ times. Here, $\epsilon_{\delta}$ is chosen in such a way that
\begin{align*}
    \mathbb P(\mathcal T_k)\leq \sum_{\mathcal C}\mathbb P\Big(\text{Bin}\Big(k,\frac{n_{\mathcal C}}{d}\Big)\leq \frac{k(1-\epsilon_{\delta})n_{\mathcal C}}{d}\Big)\leq \delta.
\end{align*}
Define 
\begin{align*}
    \eta(\mathcal C,\delta):=M(\delta)n_{\mathcal C}|\mathcal C|A^2
\end{align*} 
and choose $k$ large enough so that $\frac{k(1-\epsilon_{\delta})n_{\mathcal C}}{d}\geq \eta(\mathcal C,\delta)$ for all $\mathcal C$, which is satisfied when $k\geq \max_{\mathcal C}\frac{M(\delta)}{1-\epsilon_{\delta}}|\mathcal C|dA^2=:\alpha'A^2$. Then, Theorem \ref{MAINTHEOREM} is in force whenever $\mathcal T_k$ does not occur and $A\geq A^*(\delta, C)$ for some suitable $A^*$. Therefore, 
\begin{align*}
    d_{\infty}(K_{A,C}^{*k}(\bold p_0), \pi_{A,C})&\leq \mathbb P(\mathcal T_k)+\sup_{\mathcal C}d_{\infty}(K^{*\eta}_{A,\mathcal C}(\bold p_0^{(\mathcal C)}), \pi_{A,\mathcal C})\\&\leq 2\delta.
\end{align*}
from which the conclusion holds since $\delta$ is arbitrary.
\end{proof}

\section{Open problems and future directions}\label{sectionfuture}
\subsection{From Wasserstein to total variation.} Using the concentration and anti-concentration bounds in Lemma \ref{barycenterdidnottravel} and Lemma \ref{spreadforstat}, together with Chebyschev inequality, it is easy to upgrade the lower bound in Wasserstein distance to a lower bound of $1-\delta$ in total variation after $k=m(\delta)\frac{\lambda}{\gamma}A^2$ steps. \\ \\The situation is more complicated for the upper bound. Notice that because of conditions \eqref{choiceforH} and \eqref{conditionforeta}, our strategy only allows to push two walkers within $\delta\gg A^{-1}$ of each other. On the other hand, in order to couple even a single coordinate in total variation requires the two walkers to be closer than $O(A^{-1})$, which is the order of the variance for the single-step update in the Gibbs sampler. We believe that this is just a limitation of our method, that forces the walkers to couple in the vicinity of the corner.
\subsection{Sharper graph-dependent bounds}
The upper and lower bounds in Theorem \ref{MAINTHEOREM} are close to each other only for small perturbations of complete graphs, thus failing to provide a sharp answer to the following two interesting questions. \\ \\
A first one comes from a statistical application, when $d$ is small, but the $c_{ij}$'s are far from being equal. An example, analyzed in \cite{de1972probability}, concerns the probability of positive feedback from a certain drug in a population split into four categories: smoking men, smoking women, non-smoking men and non-smoking women (we label them $1, 2, 3, 4$ respectively). It is then reasonable to start with a prior (the weights are normalized according to Assumption \ref{assumption})
\begin{align*}
    \frac{1}{4}-2\epsilon=c_{12}=c_{34}\gg c_{13}=c_{14}=c_{23}=c_{24}=\epsilon,
\end{align*}
if we assume that smoking, rather than sex, is the main factor in the effectiveness of the drug (in their example, $\epsilon\approx 10^{-6}$). Using Theorem \ref{MAINTHEOREM} and a direct computation of the eigenvalues shows that after $k\ll \epsilon A^2$
steps the chain is not mixed yet in Wasserstein distance. On the other hand, after $k\gg A^2$ the chain is mixed. When $\epsilon$ is small, our result leaves a gap to be filled.\\ \\
Another question, more mathematically flavored, concerns family of large graphs other than the complete one. For example, for a cycle in $d$ dimensions ($d$ large), our Theorem \ref{MAINTHEOREM} shows that $d^{-2}A^2$, are needed for the mixing, while $dA^2$ suffices. This leaves a $d^3$ factor between the two bounds, raising the question of where the truth lies in between. It would be interesting to sharpen the results and understand better the relation between mixing time and geometry.
\subsection{The posterior case.} While our analysis deals with prior distributions for almost exchangeable experiments, it is interesting future work to analyze the posterior case, once data are observed. As detailed in \cite{Bacallado2015}, assume we start with a prior $\pi_{A,C}$ as defined in \eqref{stationary}, for some network $C$ on $d$ vertices, and data are sampled from some distribution. If $\bold n=(n_i)_{i\in [d]}$ data are collected in each category, and $\bold r=(r_i)_{i\in [d]}$ is the vector of positive outcomes in each category, mild assumptions on the true distributions guarantee -- for $n=\sum n_i$ large -- a posterior of the form
\begin{align*}
    \tilde \pi_{\bold n,\bold r, A, C}(\bold p)\propto \exp\{-A^2\sum_{i<j}c_{ij}(p_i-p_j)^2-nR_{\bold n,\bold r}(\bold p)\},
\end{align*}
where $R$ is the quadratic polynomial 
\begin{align*}
    R_{\bold n,\bold r}(\bold p):=\sum\limits_{i\in [d]}\Big(\frac{p_i-p^*_i}{\sigma_i}\Big)^2, \quad p^*_i:=\frac{r_i}{n_i}, \quad \sigma_i^2:=\frac{p^*_i(1-p_i^*)}{\frac{n_i}{n}}.
\end{align*}
The regime of interest for the measure $\tilde \pi_{\bold n,\bold r,A,C}$ lies, as phrased by de Finetti \cite{de1972probability}, ``in the penetration it affords into intermediate situations where the influence of initial opinion and experience balance each other". In our language, this amounts to set $n=\eta A^2$ for some fixed $\eta>0$, so that the posterior becomes of the form (we drop for convenience some of the dependencies)
\begin{align*}
    \tilde \pi_{A,Q}\propto \exp\{-A^2Q(\bold p)\},
\end{align*}
where $Q$ is the quadratic form 
\begin{align*}
    Q(\bold p):=\sum_{i<j}c_{ij}(p_i-p_j)^2+\eta R_{\bold n,\bold r}(\bold p).
\end{align*}
Unlike the prior distribution, the zero set of $Q$ is not necessarily the main diagonal. Moreover, the situation is complicated by the possible lack of convexity in $Q$. Even simple examples in two dimensions -- e.g. data coming from a unit mass at $(1,0)$ or at $(\frac{1}{2},\frac{1}{2})$ -- suggest an interesting behavior and seem worthy of further investigation.
\subsection{Interacting particles system}
As mentioned in Remark \ref{IPS}, we can view the Gibbs sampler as a dynamic for an interacting system with $d$ particles. Our analysis concerns the mixing time to stationarity in the low temperature regime, but for fixed number of particles. In this context, it is natural to address the question of taking a thermodynamical limit ($d$ increasing) for generic $A$, and the emergence of phase transitions. Even for the case of the complete graph with uniform weights, this seems an interesting and challenging problem. 
\section*{Acknowledgment}
We warmly thank Persi Diaconis for suggesting the problem being studied, and for his constant help and support. B. Gerencs\'er was supported by the J\'anos Bolyai Research Scholarship and the “Lendület” grant LP 2015-6 of the Hungarian Academy of Sciences.
 \bibliographystyle{abbrv}
  \bibliography{Bibliography.bib}
\end{document}